\documentclass{article}

\PassOptionsToPackage{square, numbers}{natbib}

     \usepackage[final]{neurips_2023}

\usepackage[utf8]{inputenc} 
\usepackage[T1]{fontenc}    
\usepackage{hyperref}       
\usepackage{url}            
\usepackage{booktabs}       
\usepackage{amsfonts}       
\usepackage{nicefrac}       
\usepackage{microtype}      
\usepackage{xcolor}         

\usepackage{caption}
\usepackage{subcaption}
\usepackage{algorithm}
\usepackage[noend]{algorithmic}
\usepackage{amsthm}
\usepackage{multirow}
\usepackage{wrapfig}
\usepackage{etoolbox}
\usepackage{amsmath}
\usepackage{etoolbox}

\newtheorem{theorem}{Theorem}[section]

\newtheorem{lemma}[theorem]{Lemma}
\newtheorem{corollary}[theorem]{Corollary}

\newif\ifAppendixlink
\newcommand{\appendixref}[1]{%
  \ifAppendixlink
    \ref{#1}%
  \else
    \ref*{#1}%
  \fi
}

\Appendixlinktrue

\usepackage{fancyhdr}

\usepackage{amsmath,amssymb}
\usepackage{graphicx}
\usepackage{accents}
\usepackage{hyperref}
\usepackage{bbm}
\usepackage{bm}

\usepackage{float}
\usepackage{scalerel}
\usepackage{tikz}
\usepackage{fancybox}

\usetikzlibrary{positioning}
\usetikzlibrary{shapes.multipart,chains,}

\newcommand{\RR}{\mathbb{R}}

\newcommand{\paren}[1]{\left( #1 \right)}
\newcommand{\br}[1]{\left[ #1 \right]}
\newcommand{\abs}[1]{\left| #1 \right|}
\newcommand{\ceil}[1]{\lceil#1\rceil}
\newcommand{\set}[1]{\{#1\}}
\newcommand{\gen}[1]{\langle #1 \rangle}

\newcommand{\norm}[1]{\left\lVert#1\right\rVert}

\newcommand{\just}[1]{\textrm{(#1)}}

\newcommand{\rt}[1]{\textrm{root}(#1)}
\newcommand{\scr}[1]{\mathcal{#1}}

\newcommand{\pluseq}{\mathrel{+}=}

\newcommand{\noiter}{\#\textrm{it}}
\newcommand{\timeseq}{\mathrel{*}=}
\newcommand{\diveq}{\mathrel{/}=}

\DeclareMathOperator*{\startimes}{\scalerel*{\circledast}{\sum}}

\DeclareMathOperator*{\argmin}{arg\,min}
\DeclareMathOperator*{\diag}{diag}

\newcommand\Mark[1]{\textsuperscript{#1}}

\title{Fast Exact Leverage Score Sampling from Khatri-Rao Products
with Applications to Tensor Decomposition}

\author{%
  Vivek Bharadwaj \Mark{1,2}, Osman Asif Malik \Mark{2}, Riley Murray
  \Mark{3,2,1}, \\
  \bf{Laura Grigori \Mark{4}, Ayd{\i}n Bulu\c{c} \Mark{2, 1}, 
  James Demmel \Mark{1}}\\
  \Mark{1}Electrical Engineering and Computer Science Department, UC Berkeley\\
  \Mark{2}Computational Research Division, Lawrence Berkeley National Lab\\
  \Mark{3}International Computer Science Institute \\
  \Mark{4} Institute of Mathematics, EPFL \& Lab for Simulation and Modelling, Paul Scherrer Institute 
}

\begin{document}

\maketitle

\begin{abstract}
We present a data structure 
to randomly sample rows from the Khatri-Rao
product of several matrices according to the exact distribution of its leverage scores. 
Our proposed sampler draws each row in time logarithmic in the height of 
the Khatri-Rao product and quadratic in its column count, with persistent space overhead
at most the size of the input matrices. As a result, it tractably draws samples 
even when the matrices forming the Khatri-Rao product have tens of millions of rows each. When used to sketch the linear least squares problems arising in CANDECOMP / PARAFAC 
tensor decomposition, our method achieves lower asymptotic complexity 
per solve than recent state-of-the-art methods. Experiments on billion-scale sparse tensors validate our claims, with our algorithm 
achieving higher accuracy than competing methods as the 
decomposition rank grows.
\end{abstract}

\section{Introduction}
\label{sec:introduction}
The Khatri-Rao product (KRP, denoted by $\odot$) is the 
column-wise Kronecker product of two matrices, and
it appears in diverse applications across numerical analysis and machine learning
\cite{krp_survey}. We examine overdetermined linear
least squares problems of the form $\min_X \norm{AX - B}_F$, where the design 
matrix $A = U_1 \odot ... \odot U_N$ 
is the Khatri-Rao product of matrices $U_j \in \RR^{I_j \times R}$.
These problems appear prominently in signal 
processing \cite{krp_signal_processing}, 
compressed sensing \cite{mimo_radar_krp}, 
inverse problems related to partial differential
equations \cite{pde_inverse_sketching}, 
and alternating least squares (ALS)
CANDECOMP / PARAFAC (CP) tensor decomposition \cite{kolda_tensor_overview}.
In this work, we focus on the case where $A$ has moderate column count
(several hundred at most). Despite this, the problem remains formidable 
because the height of $A$ is $\prod_{j=1}^N I_j$. For
row counts $I_j$ in the millions, it is intractable 
to even materialize $A$ explicitly.

Several recently-proposed randomized sketching algorithms can approximately 
solve least squares problems with Khatri-Rao product design matrices
\cite{battaglino_practical, jin_faster_2020, larsen_practical_2022,malik_efficient_2022, woodruff_zandieh}. 
These methods apply a sketching operator $S$ to the design 
and data matrices to solve the reduced least squares problem
$\min_{\tilde X} \norm{SA \tilde X - SB}_F$, where $S$ has far fewer rows than columns. For appropriately chosen $S$, the residual of the downsampled system 
falls within a specified tolerance $\varepsilon$ of the optimal residual with high probability $1 - \delta$. 
In this work, we constrain $S$ to be a \textit{sampling matrix} that 
selects and reweights a subset of rows from both $A$ and $B$. When the 
rows are selected according to the
distribution of \textit{statistical leverage scores} on the design matrix $A$,
only $\tilde O\paren{R / (\varepsilon \delta)}$ samples are required (subject
to the assumptions at the end of section \ref{sec:sketched_linear_lstsq}). 
The challenge, then, is to efficiently sample according to the 
leverage scores when $A$ has Khatri-Rao structure.

We propose a leverage-score sampler for the Khatri-Rao
product of matrices with tens of millions of rows each. 
After construction, our sampler draws each row in time quadratic 
in the column count, but logarithmic in the total row count of the Khatri-Rao product. 
Our core contribution is the following theorem.

\begin{theorem}[Efficient Khatri-Rao Product Leverage Sampling]
Given $U_1, ..., U_N$ with $U_j \in \RR^{I_j \times R}$, 
there exists a data structure satisfying the following: 
\begin{enumerate}
    \item The data structure has construction time $O\paren{\sum_{j=1}^{N} 
    I_j R^2}$
    and requires additional storage space $O\paren{\sum_{j=1}^{N} I_j R}.$ If a single entry in a
    matrix $U_j$ changes, it can be updated in time $O(R \log \paren{I_j / R})$. 
    If the entire matrix $U_j$ changes, it can be updated in time $O\paren{I_j R^2}$.

    \item The data structure produces $J$ samples from the Khatri-Rao product $U_1 \odot ... \odot U_N$ according to the exact leverage score distribution on its rows in time 
    \[
    O\paren{NR^3 + J \sum_{k=1}^N R^2 \log \max \paren{I_k, R}} 
    \]  
    using $O(R^3)$ scratch space. The structure can also draw samples from the Khatri-Rao product of any subset of $U_1, ..., U_N$. 
\end{enumerate}
\label{thm:main_krp_res}
\end{theorem}
The efficient update property and ability to exclude one matrix are important in CP decomposition. When the inputs
$U_1, ..., U_N$ are sparse, an analogous
data structure with 
$O\paren{R \sum_{j=1}^N \textrm{nnz}(U_j)}$ 
construction time 
and $O\paren{\sum_{j=1}^N \textrm{nnz}(U_j)}$ storage space exists with identical
sampling time. Since the output factor matrices $U_1, ..., U_N$ 
are typically dense, we defer the proof to Appendix 
\appendixref{appendix:sparse_input_extension}. 
Combined with error
guarantees for leverage-score sampling, we achieve an algorithm for alternating least squares CP decomposition
with asymptotic complexity lower than recent state-of-the-art methods (see Table 
\ref{tab:cp-complexity-comparison}).

Our method provides the most practical benefit on
sparse input tensors, which may have dimension lengths in the 
tens of millions (unlike dense tensors that quickly incur intractable
storage costs at large dimension lengths) \cite{frosttdataset}. 
On the Amazon and Reddit tensors with billions of nonzero entries, 
our algorithm STS-CP can achieve 95\% of the fit of non-randomized ALS 
between 1.5x and 2.5x faster than a high-performance implementation 
of the state-of-the-art CP-ARLS-LEV 
algorithm \cite{larsen_practical_2022}. 
Our algorithm is significantly more sample-efficient; on 
the Enron tensor, only $\sim 65,000$ samples per solve were required 
to achieve the 95\% accuracy threshold above a rank of 50, which could 
not be achieved by CP-ARLS-LEV with even 54 times as many samples.

\begin{table}[htb]
\caption{Asymptotic Complexity to decompose an $N$-dimensional $I \times ... \times I$ 
dense tensor via CP alternating least squares. For randomized
algorithms, each approximate least-squares solution has residual
within $(1 + \varepsilon)$ of the optimal value with high probability
$1 - \delta$. Factors involving $\log R$ and $\log (1 / \delta)$ are hidden ($\tilde O$ notation). See \appendixref{appendix:intro} for details.} 
\label{tab:cp-complexity-comparison}
\begin{center}
\begin{small}
\begin{sc}
	\begin{tabular}{ll}  
		\toprule
		Algorithm & Complexity per Iteration 												\\
		\midrule
		CP-ALS \cite{kolda_tensor_overview}		& $N (N+I) I^{N-1} R$ 						\\ 
		CP-ARLS-LEV	\cite{larsen_practical_2022} & $N ( R + I ) R^{N} / (\varepsilon \delta)$	\\ 
		TNS-CP \cite{malik_tns_cp} & $N^3 I R^3 / (\varepsilon \delta)$ 			\\ 
		Gaussian TNE \cite{ma_cost_efficient_embedding}  & $N^2(N^{1.5} R^{3.5} / \varepsilon^3 + I R^2) / \varepsilon^2$ 			\\ 
		\textbf{STS-CP (ours)}	   			& $N (N R^3 \log I + IR^2) / (\varepsilon \delta)$ 			\\
		\bottomrule
	\end{tabular} 
\end{sc}
\end{small}
\end{center}
\end{table}

\section{Preliminaries and Related Work}
\label{sec:preliminaries}
\paragraph{Notation.}
We use $\br{N}$ to denote the set $\set{1, ..., N}$ for a
positive integer $N$.
We use $\tilde O$ notation to indicate the presence of multiplicative 
terms polylogarithmic in $R$ and $(1 / \delta)$ in runtime complexities.
For the complexities of our methods, these logarithmic factors are 
no more than $O(\log (R / \delta))$. We 
use Matlab notation $A\br{i, :},
A\br{:, i}$ to index rows, resp. columns, of matrices. For consistency, we
use the convention that $A\br{i, :}$ is a row vector. 
We use $\cdot$ for standard matrix multiplication, $\circledast$ as 
the elementwise product, $\otimes$ to
denote the Kronecker product, and $\odot$ for the Khatri-Rao product. 
See Appendix \appendixref{appendix:notation} for a definition of each operation.
Given matrices $A \in \RR^{m_1 \times n}, B \in \RR^{m_2 \times n}$,
the $j$-th column of the Khatri-Rao product 
$A \odot B \in \RR^{m_1 m_2 \times n}$ is the
Kronecker product $A\br{:, j} \otimes B\br{:, j}$.

We use angle brackets $\gen{\cdot, ..., \cdot}$ to denote a \textbf{generalized inner product}.
For identically-sized vectors / matrices,
it returns the sum of all entries in 
their elementwise product. For 
$A, B, C \in \RR^{m \times n}$, 
\[
\gen{A, B, C} := \sum_{i=1,j=1}^{m, n} A\br{i, j} B\br{i, j} C\br{i, j}.
\]
Finally, $M^+$ denotes the pseudoinverse of matrix $M$.

\subsection{Sketched Linear Least Squares}
\label{sec:sketched_linear_lstsq}
A variety of random sketching operators $S$ have been 
proposed to solve overdetermined least squares problems $\min_X \norm{AX - B}_F$ when $A$ has no special structure 
\cite{woodruff2014sketching,fjlt}. When $A$ has
Khatri-Rao product structure, prior work has focused on 
\textit{sampling} matrices \cite{cheng_spals_2016, larsen_practical_2022}, which have a single nonzero entry per row, 
operators composed of fast Fourier /
trigonometric transforms \cite{jin_faster_2020}, or Countsketch-type operators \cite{wang2015fast, ahle_treesketch}. 
For tensor decomposition, however, the matrix $B$ may be sparse or implicitly specified
as a black-box function. When $B$ is sparse, Countsketch-type operators 
still require the algorithm to iterate over all nonzero values in $B$. As \citet{larsen_practical_2022} note, operators similar to the FFT 
induce fill-in when applied to a sparse matrix $B$, destroying
the benefits of sketching.
Similar difficulties arise when $B$ is implicitly specified. This motivates our
decision to focus on row sampling operators, which only touch a subset of entries from
$B$. Let $\hat x_1, ..., \hat x_J$ be a selection of $J$ indices for the rows
of $A \in \RR^{I \times R}$, sampled i.i.d.\ according to a probability distribution
$q_1, ..., q_I$. The associated sampling matrix $S \in \RR^{J \times I}$ is 
specified by
\[
S\br{j, i}= 
\begin{cases}
    \frac{1}{\sqrt{J q_i}},& \text{if } \hat x_j = i \\
    0,              & \text{otherwise}
\end{cases}
\]
where the weight of each nonzero entry corrects bias induced by sampling. 
When the probabilities $q_j$ are proportional to the
\textit{leverage scores} of the rows of $A$, strong guarantees apply to the
solution of the downsampled problem. 

\paragraph{Leverage Score Sampling.} 
The leverage scores of a matrix assign a measure of importance to each of its
rows. The leverage score of row $i$ from matrix $A \in \RR^{I \times R}$ is given by
\begin{equation}
\ell_i = A\br{i, :} (A^\top A)^+ A \br{i, :}^\top
\label{eq:leverage_definition}
\end{equation}
for $1 \leq i \leq I$. Leverage scores can 
be expressed equivalently as the squared row norms of the 
matrix $Q$ in any reduced $QR$ factorization of $A$
\cite{drines_fast_leverage_scores}.
The sum of all leverage scores is the rank of $A$ \cite{woodruff2014sketching}. 
Dividing the scores by their sum, we induce a probability distribution on the rows used
to generate a sampling matrix $S$. The next theorem has appeared in
several works, and we take the form given by 
\citet{malik_tns_cp}. For an appropriate sample count, it guarantees that
the residual of the downsampled problem is close to the residual of the 
original problem.

\begin{theorem}[Guarantees for Leverage Score Sampling] 
Given $A \in \RR^{I \times R}$ and $\varepsilon, \delta \in (0, 1)$, 
let $S \in \RR^{J \times I}$ be a leverage score sampling matrix
for $A$. 
Further define $\tilde X = \argmin_X \norm{SAX - SB}_F$. 
If $J \gtrsim
R \max \paren{\log \paren{R / \delta}, 1 / (\varepsilon \delta)}$, then
with probability at least $1 - \delta$ it holds that
\[
\norm{A \tilde X - B}_F \leq (1 + \varepsilon) \min_X \norm{AX - B}_F.
\]
\label{thm:lev_score_lowerbounds}
\end{theorem}
For the applications considered in this work, $R$ ranges up to a few hundred. 
As $\varepsilon$ and $\delta$ tend to 0 with fixed $R$, 
$1 / (\varepsilon \delta)$ dominates $\log(R / \delta)$. Hence, 
we assume that the minimum sample count $J$ to achieve the
guarantees of the theorem is $\Omega(R / (\varepsilon \delta))$. 

\subsection{Prior Work}
\paragraph{Khatri-Rao Product Leverage Score Sampling.} 
Well-known sketching algorithms exist to quickly 
estimate the leverage scores of dense matrices
\cite{drines_fast_leverage_scores}. These algorithms
are, however, intractable for 
$A = U_1 \odot ... \odot U_N$ 
due to the height of the Khatri-Rao product.
\citet{cheng_spals_2016} instead approximate each score 
as a product of leverage scores associated with each matrix $U_j$.
\citet{larsen_practical_2022} 
propose CP-ARLS-LEV, which uses a similar approximation 
and combines random sampling with a deterministic selection of high-probability indices. Both methods were presented in the context
of CP decomposition. To sample from the
Khatri-Rao product of $N$ matrices,
both require $\tilde O(R^N / (\varepsilon \delta))$ samples 
to achieve the $(\varepsilon, \delta)$ guarantee on the residual of each
least squares solution. These methods are simple to implement and perform well
when the Khatri-Rao product has column count up to 20-30. On the other hand,
they suffer from high sample complexity as $R$ and $N$ increase. The TNS-CP algorithm by 
\citet{malik_tns_cp} samples from the exact leverage score distribution, thus requiring only 
$\tilde O(R / (\varepsilon \delta))$ samples per 
least squares solve. Unfortunately, it requires time $O\paren{\sum_{j=1}^N I_j R^2}$ 
to draw each sample. 

\paragraph{Comparison to Woodruff and Zandieh.}
The most comparable results to ours appear in work by
\citet{woodruff_zandieh}, 
who detail an algorithm for approximate ridge leverage-score sampling 
for the Khatri-Rao product in near input-sparsity time.
Their work relies on a prior oblivious method by
\citet{ahle_treesketch}, which sketches a Khatri-Rao 
product using a sequence of 
Countsketch / OSNAP operators arranged in a tree. 
Used in isolation to solve a linear least squares problem, the 
tree sketch construction time 
scales as $O \paren{\frac{1}{\varepsilon}
\sum_{j=1}^N \textrm{nnz}(U_j)}$ and requires an embedding 
dimension quadratic in $R$ to achieve the $(\varepsilon, \delta)$ 
solution-quality guarantee. Woodruff and Zandieh use a collection of
these tree sketches, each with carefully-controlled approximation error,
to design an algorithm with 
linear runtime dependence on the column count $R$.
On the other hand,
the method exhibits $O(N^7)$ scaling in the number of matrices involved, has $O(\varepsilon^{-4})$ scaling 
in terms of the desired accuracy, and relies on a sufficiently high ridge regularization parameter. 
Our data structure instead requires construction time quadratic in $R$. 
In exchange, we use distinct methods to design an efficiently-updatable sampler with 
runtime linear in both $N$ and $\varepsilon^{-1}$. These properties are attractive
when the column count $R$ is below several thousand and when error as low 
as $\epsilon \approx 10^{-3}$ is needed in the context of an iterative solver 
(see Figure \ref{fig:exact_solve_comparison_sparse}). 
Moreover, the term $O(R^2 \sum_{j=1}^N I_j)$ in our construction complexity 
arises from symmetric rank-$k$ updates, a highly-optimized BLAS3 kernel on modern CPU and GPU architectures. Appendix 
\appendixref{appendix:wz_comparison} provides a more detailed comparison between the two approaches.

\paragraph{Kronecker Regression.} Kronecker regression is a distinct (but closely related) problem to the
one we consider. There, $A = U_1 \otimes ... \otimes U_N$ and 
the matrices $U_i$ have
potentially distinct column counts $R_1, ..., R_N$. While 
the product distribution of leverage scores from $U_1, ..., U_N$
provides only an
approximation to the leverage score distribution of the 
Khatri-Rao product
\cite{cheng_spals_2016, larsen_practical_2022}, it provides the 
\textit{exact} leverage distribution for the Kronecker product.
Multiple works \cite{woodruff_optimal_kronecker, subquadratic_kronecker_regression} combine this property with other
techniques, such as dynamically-updatable 
tree-sketches \cite{dynamic_kron_regression}, to produce accurate
and updatable Kronecker sketching methods. None of these results
apply directly in our case due to the distinct properties of Kronecker and Khatri-Rao products. 

\section{An Efficient Khatri-Rao Leverage Sampler}
\label{sec:efficient_krp_sampling}
Without loss of generality, we will prove part 2 of Theorem
\ref{thm:main_krp_res} for the case where $A = U_1 \odot ... \odot U_N$; the
case that excludes a single matrix follows by reindexing matrices
$U_k$. We further assume that $A$ is a nonzero matrix, though it may be
rank-deficient. Similar to prior sampling works 
\cite{malik_efficient_2022, woodruff_zandieh}, our 
algorithm will draw one sample from
the Khatri-Rao product by sampling a row from 
each of $U_1, U_2, ....$ in sequence and computing their
Hadamard product, with the draw from 
$U_j$ conditioned on prior draws from $U_1, ..., U_{j-1}$. 

Let us index each row of $A$ by a tuple 
$(i_1, ..., i_{N}) \in \br{I_1} \times ... \times \br{I_{N}}$.
Equation \eqref{eq:leverage_definition} gives
\begin{equation}
\ell_{i_1, ..., i_{N}} =A \br{(i_1, ..., i_{N}), :} (A^\top A)^+ A \br{(i_1, ..., i_{N}), :}^\top.
\label{eq:krp_leverage}
\end{equation}
For $1 \leq k \leq N$, define $G_k := U_k^\top U_k \in \RR^{R \times R}$ and 
$G := \paren{\startimes_{k=1}^{N} G_k} \in \RR^{R \times R}$; it is a well-known fact that
$G = A^\top A$ \cite{kolda_tensor_overview}. For a single row sample from $A$, 
let $\hat s_1, ..., \hat s_{N}$ be random variables for 
the draws from multi-index set $\br{I_1} \times ... \times \br{I_N}$ 
according to the leverage score distribution. Assume, for some $k$, that we have 
already sampled an index 
from each of $\br{I_1}, ..., \br{I_{k-1}}$, and that 
the first $k-1$ random variables take values 
$\hat s_1 = s_1, ..., \hat s_{k-1} = s_{k-1}$. We
abbreviate the latter condition
as $\hat s_{< k} = s_{< k}$. 
To sample from $I_k$, we seek the distribution of 
$\hat s_k$ conditioned on $\hat s_1, ... \hat s_{k-1}$. 
Define $h_{< k}$ as the transposed 
elementwise product\footnote{For 
$a > b$, assume that $\startimes_{i=a}^b \paren{...}$ produces 
a vector / matrix filled with ones.} of rows 
already sampled:
\begin{equation}
h_{< k} := \startimes_{i=1}^{k-1} U_i \br{s_i, :}^\top.
\label{eq:h_leq_k_defn}
\end{equation}
Also define $G_{>k}$ as
\begin{equation}
G_{> k} := G^+ \circledast \startimes_{i=k+1}^N G_i.
\label{eq:G_geq_k_defn}
\end{equation}
Then the following theorem provides the conditional distribution of $\hat s_k$.
\begin{theorem}[Malik 2022, \cite{malik_efficient_2022}, Adapted]
For any $s_k \in \br{I_k}$,
\begin{equation}
\begin{aligned}
p(\hat s_k = s_k\ \vert\ \hat s_{< k} = s_{< k}) 
&= C^{-1} \gen{h_{<k} h_{<k}^\top, U_k\br{s_k, :}^\top U_k \br{s_k, :}, G_{> k}} \\
&:= q_{h_{<k} , U_k ,G_{>k}}\br{s_k}
\label{eq:malik_equation}
\end{aligned}
\end{equation}
where $C = \gen{h_{<k} h_{<k}^\top, U_k^\top U_k, G_{>k}}$ is nonzero.
\label{thm:malik2022}
\end{theorem}
We include the derivation of 
Theorem \ref{thm:malik2022} from Equation 
\eqref{eq:krp_leverage} in Appendix \appendixref{appendix:malik_reproof}. 
Computing all entries of the probability vector $q_{h_{<k} , U_k ,G_{>k}}$
would cost $O(I_j R^2)$ per sample, too costly when $U_j$ has
millions of rows. It is likewise 
intractable (in preprocessing time and space complexity) to 
precompute probabilities for every possible conditional distribution
on the rows of $U_j$, since the conditioning random variable has 
$\prod_{k=1}^{j-1} I_k$ potential values. Our key innovation is a
data structure to sample from a discrete distribution of
the form $q_{h_{<k} , U_k ,G_{>k}}$ \textit{without} materializing all of its
entries or incurring superlinear cost in either $N$ or $\varepsilon^{-1}$. 
We introduce this data structure in the next section and will apply it 
twice in succession to get the complexity in Theorem
\ref{thm:main_krp_res}.

\subsection{Efficient Sampling from $q_{h, U, Y}$}
\label{sec:row_sampling_problem}
We introduce a slight change of notation in this section to simplify the problem 
and generalize our sampling lemma. Let $h \in \RR^{R}$ be a vector
and let $Y \in \RR^{R \times R}$ be a positive semidefinite (p.s.d.) matrix, respectively. 
Our task is to sample $J$ rows from a matrix 
$U \in \RR^{I \times R}$ according to the distribution 
\begin{equation}
q_{h, U, Y}\br{s}
:= 
C^{-1} \gen{h h^\top, U^\top \br{s, :} U\br{s, :}, Y}
\end{equation}
provided the normalizing constant $C = \gen{h h^\top, U^\top U, Y}$, 
is nonzero. We impose that all $J$ rows are drawn 
with the same matrices $Y$ and $U$, 
but potentially distinct vectors $h$. The following lemma establishes that
an efficient sampler for this problem exists.
\begin{lemma}[Efficient Row Sampler]
Given matrices $U \in \RR^{I \times R}, Y \in \RR^{R \times R}$ with $Y$ p.s.d., there exists a data
structure parameterized by positive integer $F$ 
that satisfies the following: 
\begin{enumerate}
    \item The structure has construction time 
    $O\paren{IR^2}$ and storage requirement $O\paren{R^2 
    \ceil{I/F}}$. If $I < F$, the storage requirement drops to $O(1)$.

    \item After construction, the data structure can produce a sample according to the
    distribution $q_{h, U, Y}$ in time $O(R^2 \log \ceil{I / F} + F R^2)$ for any vector  
    $h$.

    \item If $Y$ is a rank-1 matrix, 
    the time per sample drops to
    $O(R^2 \log \ceil{I / F} + F R)$.
\end{enumerate}
\label{lemma:efficient_row_sampler}
\end{lemma}

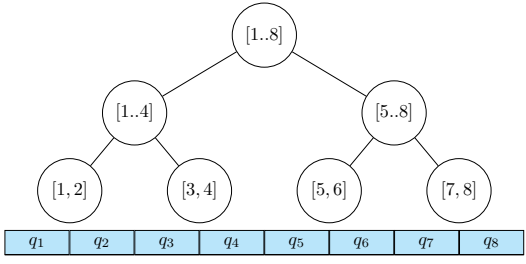
\begin{wrapfigure}{R}{0.5\textwidth}
    \begin{minipage}{0.5\textwidth}

\scalebox{0.69}{
\begin{tikzpicture}[level/.style={sibling distance=50mm/#1},
    start chain=going right,node distance=0mm, minimum width=3.52em
    ]
    \node [circle, draw] (r){$\br{1..8}$}
        child {node [circle, draw] (v1){$\br{1..4}$}
            child {node [circle, draw] (v3){$\br{1,2}$}} 
            child {node [circle, draw] (v4){$\br{3,4}$}}
        } 
        child {node [circle, draw] (v2){$\br{5..8}$} 
            child {node [circle, draw] (v5){$\br{5,6}$}} 
            child {node [circle, draw] (v6){$\br{7,8}$}} 
        }
        ;

    \tikzset{
      token/.style={
        rectangle split,
        rectangle split parts=2,
        rectangle split part fill={cyan!25,yellow!50},
        rectangle split ignore empty parts,
        draw,
      },
    };
  \node[on chain,token] at (-4.38, -4.0) {$q_1$}; 
  \node[on chain,token] {$q_2$}; 
  \node[on chain,token] {$q_3$}; 
  \node[on chain,token] {$q_4$}; 
  \node[on chain,token] {$q_5$};
  \node[on chain,token] {$q_6$}; 
  \node[on chain,token] {$q_7$}; 
  \node[on chain,token] {$q_8$};
        
\end{tikzpicture}
}
    \captionof{figure}{A segment tree $T_{8,2}$ and probability distribution $\set{q_1, ..., q_8}$ on 
    $\br{1, ..., 8}$.}
    \label{fig:segment_tree}    
    \end{minipage}
\end{wrapfigure}
This data structure relies on an adaptation of a 
classic binary-tree inversion sampling technique \cite{saad_optimal_2020}. Consider a 
partition of the interval $\br{0, 1}$ into $I$ bins, the $i$-th having width 
$q_{h, U, Y}\br{i}$. We sample $d \sim \textrm{Uniform}\br{0, 1}$ and return the index of the containing bin. We locate the bin index 
through a binary search terminated when at most $F$ bins remain in 
the search space, which are then scanned in linear time. Here, $F$ is a tuning parameter 
that we will use to control sampling complexity and space usage.

We can regard the binary search as a walk down a full, complete binary tree $T_{I, F}$ with
$\ceil{I / F}$ leaves, the nodes of which store contiguous, disjoint segments 
$S(v) = \set{S_0(v)..S_1(v)} \subseteq \br{I}$ of size at most $F$. The segment of each internal
node is the union of segments held by its children, and the root node
holds $\set{1,...,I}$. Suppose that the binary search
reaches node $v$ with left child $L(v)$ and maintains the interval 
$\br{\textrm{low}, \textrm{high}} \subseteq \br{0, 1}$ as the remaining 
search space to explore. 
Then the search branches left in the tree iff
$d < \textrm{low} + \sum_{i \in S(L(v))} q_{h, U, Y} \br{i}.$ 

This branching condition can be evaluated efficiently if appropriate information
is stored at each node of the segment tree. Excluding the offset ``low'', 
the branching threshold takes the form
\begin{equation}
\sum_{i \in S(v)} q_{h, U, Y} \br{i} 
= C^{-1} \gen{h h^\top, \sum_{i \in S(v)} U\br{i, :}^\top U\br{i, :}, Y} 
:=C^{-1} \gen{h h^\top, G^v, Y}.
\label{eq:partial_gram}
\end{equation}
Here, we call each matrix $G^v \in \RR^{R \times R}$ a \textit{partial Gram matrix}. In time $O(I R^2)$ and space $O(R^2 \ceil{I / F})$,
we can compute and cache $G^v$ for each node of the tree to construct our
data structure. Each subsequent binary search costs $O(R^2)$ time 
to evaluate Equation \eqref{eq:partial_gram} at each of
$\log \ceil{I / F}$ internal nodes and $O(FR^2)$ to
evaluate $q_{h, U, Y}$ at the $F$ indices held by each leaf, giving point
2 of the lemma. This cost at each leaf node reduces to $O(FR)$ in case 
$Y$ is rank-1, giving point 3. A complete proof of this lemma appears in 
Appendix \appendixref{appendix:efficient_lemma_proof}.

\subsection{Sampling from the Khatri-Rao Product}
\label{sec:krp_leverage_defs}
We face difficulties if we directly apply 
Lemma \ref{lemma:efficient_row_sampler} to sample
from the conditional distribution in Theorem \ref{thm:malik2022}. 
Because $G_{>k}$ is not rank-1 in general, we must use point 2 of the lemma, where 
no selection of the parameter $F$ allows us to simultaneously satisfy the space and runtime constraints of Theorem \ref{thm:main_krp_res}. Selecting $F = R$ results in cost 
$O(R^3)$ per 
sample (violating the runtime requirement in point 2), 
whereas $F=1$ 
results in a superlinear storage requirement $O(IR^2)$ (violating 
the space requirement in point 1, and becoming prohibitively
expensive for $I \geq 10^6$). To avoid these extremes, 
we break the 
sampling procedure into two stages. The
first stage selects a 1-dimensional subspace spanned by an eigenvector of $G_{>k}$, 
while the second samples according to Theorem \ref{thm:malik2022} after projecting
the relevant vectors onto the selected subspace. 
Lemma \ref{lemma:efficient_row_sampler} can be used for \textit{both} stages, 
and the second stage benefits from point 3 to achieve 
better time and space complexity. 

Below, we abbreviate $q = q_{h_{<k}, U_k, G_{>k}}$ and $h = h_{<k}$. 
When sampling from $I_k$, observe that 
$G_{>k}$ is the same for all samples. We compute a symmetric eigendecomposition $G_{>k} = V \Lambda V^\top$, where each column of $V$ is an eigenvector of $G_{>k}$ and $\Lambda = \diag((\lambda_u)_{u=1}^R)$ contains the eigenvalues along the diagonal. 
This allows us to rewrite entries of $q$ as 
\begin{equation}
q\br{s} = C^{-1} \sum_{u=1}^R \lambda_u \gen{h h^\top, U_k\br{s, :}^\top U_k \br{s, :}, V\br{:, u} V\br{:, u}^\top}. 
\end{equation}
Define matrix $W \in \RR^{I_k \times R}$ elementwise by 
\[
W\br{t, u} := \gen{h h^\top, U_k\br{t, :}^\top U_k \br{t, :}, V\br{:, u} 
V\br{:, u}^\top} \\
\]
and observe that all of its entries are nonnegative. 
Since $\lambda_u \geq 0$ for all $u$ ($G_{>k}$ is
p.s.d.), we can write $q$ as a mixture of probability 
distributions given by the normalized columns of $W$:
\[
q= \sum_{u=1}^R w\br{u} \frac{W\br{:, u}}{\norm{W\br{:, u}}_1},
\]
where the vector $w$ of nonnegative weights is given by $w\br{u} = (C^{-1} \lambda_u \norm{W \br{:, u}}_1)$. Rewriting $q$ in this form gives us the two stage sampling procedure: first sample a 
component $u$ of the mixture according to the weight vector $w$, then sample an 
index in $\set{1..I_k}$ according to the probability vector 
defined by $W\br{:, u} / \norm{W\br{:, u}}_1$. Let $\hat u_k$ be a random variable distributed according to the 
probability mass vector $w$. We have, for $C$ taken from Theorem \ref{thm:malik2022}, 
\begin{equation}
\begin{aligned}
p(\hat u_k = u_k) &=  C^{-1} \lambda_{u_k}
\sum_{t=1}^{I_k} W\br{t, u_k} \\
&= C^{-1} \lambda_{u_k} \gen{h h^\top, V\br{:, u_k} V\br{:, u_k}^\top, G_k} \\
&= q_{h, \sqrt{\Lambda} V^\top, G_k} \br{u_k}.
\label{eq:first_dist}
\end{aligned}
\end{equation}
Hence, we can use point 2 of Lemma 
\ref{lemma:efficient_row_sampler} to sample a value for 
$\hat u_k$ efficiently. Because $\sqrt{\Lambda} V^\top$ has 
only $R$ rows with $R \sim 10^2$, we can choose tuning parameter $F = 1$
to achieve lower time per sample while incurring a modest $O(R^3)$ space
overhead. Now, introduce a random variable 
$\hat t_k$ with distribution conditioned on $\hat u_k = u_k$ 
given by
\begin{equation}
p(\hat t_k = t_k\ \vert\ \hat u_k = u_k) := W\br{t_k, u_k} / \norm{W \br{:, u_k}}_1.
\end{equation}
This distribution is well-defined, since we suppose that $\hat u_k = u_k$ occurs with
nonzero probability $e \br{u_k}$, which implies that $\norm{W \br{:, u_k}}_1 \neq 0$. 
Our remaining task is to efficiently sample from the distribution above.
Below, we abbreviate $\tilde h = V \br{:, u_k} \circledast h$ and
derive
\begin{equation}
\begin{aligned}
p(\hat t_k = t_k\ \vert\ \hat u_k = u_k) &= \frac{\gen{h h^\top, U_k \br{t_k, :}^\top U_k \br{t_k, :}, V\br{:, u_k} V\br{:, u_k}^\top}}{\norm{W\br{:, u_k}}_1} \\
&= \frac{\gen{\tilde h \tilde h^\top, U_k \br{t_k, :}^\top U_k \br{t_k, :}, \br{1}}}{\norm{W\br{:, u_k}}_1} \\
&= q_{\tilde h, U_k, \br{1}}\br{t_k}.
\label{eq:second_conditional_dist}
\end{aligned}
\end{equation}
Based on the last line of Equation \eqref{eq:second_conditional_dist}, we 
apply Lemma \ref{lemma:efficient_row_sampler} again to build an 
efficient data structure to sample a row of 
$U_k$. Since $Y = \br{1}$ is a rank-1 matrix, 
we can use point 3 of the lemma and 
select a larger parameter value $F = R$ to reduce
space usage. The sampling time for this stage
becomes $O(R^2 \log \ceil{I_j / R}).$ 

To summarize, Algorithms \ref{alg:krp_sampler_construction}
and \ref{alg:krp_sampler_sampling} give the construction and sampling
procedures for our data structure. They rely on the
``BuildSampler" and ``RowSample" procedures from Algorithms
\ref{alg:row_sampler_construction} and 
\ref{alg:row_sampler_sampling} in Appendix 
\appendixref{appendix:efficient_lemma_proof}, which relate 
to the data structure in Lemma \ref{lemma:efficient_row_sampler}.
In the construction phase, we build
$N$ data structures from Lemma \ref{lemma:efficient_row_sampler}
for the distribution in Equation \eqref{eq:second_conditional_dist}.
Construction costs $O\paren{\sum_{j=1}^N I_j R^2}$, and if 
any matrix $U_j$ changes, we can rebuild $Z_j$ in isolation. Because $F = R$, the space required for $Z_j$ is 
$O\paren{I_j R}$. In the sampling phase, the procedure in Algorithm 
\ref{alg:krp_sampler_sampling} accepts an 
optional index $j$ of a matrix to exclude from the Khatri-Rao product.
The procedure begins by computing the symmetric eigendecomposition of
each matrix $G_{>k}$. The eigendecomposition is computed only 
once per binary tree structure, and its computation cost 
is amortized over all $J$ samples. It then creates 
data structures $E_k$ for each of the distributions specified by
Equation \eqref{eq:first_dist}. These data structures 
(along with those
from the construction phase) are used to draw $\hat u_k$ and
$\hat t_k$ in succession. The random variables $\hat t_k$ 
follow the distribution in Theorem \ref{thm:malik2022} conditioned on 
prior draws, so the multi-index $(\hat t_k)_{k \neq j}$ 
follows the leverage score distribution on $A$, as desired. 
Appendix \appendixref{appendix:cohesive_krp_proof} 
proves the complexity claims in the theorem and provides further details about the algorithms.

\begin{wrapfigure}{R}{0.5\textwidth}
    \begin{minipage}{0.5\textwidth}
\begin{algorithm}[H]
    \caption{ConstructKRPSampler($U_1, ..., U_N$)}
    \begin{algorithmic}[1]
    \FOR{$j =1..N$}
        \STATE $Z_j := \textrm{BuildSampler}(U_j, F=R, \br{1})$
        \STATE $G_j := U_j^\top U_j$
    \ENDFOR
    \end{algorithmic}
    \label{alg:krp_sampler_construction}
\end{algorithm}
\vspace{-2em}
\begin{algorithm}[H]
    \caption{KRPSample($j$, $J$)}
    \begin{algorithmic}[1]
    \STATE $G := \startimes_{k \neq j} G_k$ 
    \FOR{$k \neq j$}
        \STATE $G_{>k} := G^+ \circledast \startimes_{k=j+1}^N G_k$
        \STATE Decompose $G_{>k} = V_k \Lambda_k V_k^\top$
        \STATE $E_k := \textrm{BuildSampler}(\sqrt{\Lambda_k} \cdot V_k^\top, F=1, G_k)$ 
    \ENDFOR
    \FOR{$d =1..J$}
        \STATE $h = \br{1, ..., 1}^\top$
        \FOR{$k \neq j$}
            \STATE $\hat u_k := \textrm{RowSample}(E_k, h)$ 
            \STATE $\hat t_k := \textrm{RowSample}(Z_k, h \circledast (V_k \br{:, \hat u_k}))$
            \STATE $h \timeseq U_k \br{\hat t_k, :}$
        \ENDFOR
        \STATE $s_d = (\hat t_k)_{k \neq j}$
    \ENDFOR
    \STATE \textbf{return} $s_1, ..., s_J$
    \end{algorithmic}
    \label{alg:krp_sampler_sampling}
\end{algorithm}     
    \end{minipage}
  \end{wrapfigure}  

\subsection{Application to Tensor Decomposition}
A tensor is a multidimensional array, and the CP decomposition represents a tensor as 
a sum of outer products 
\cite{kolda_tensor_overview}. See Appendix
\appendixref{appendix:als_cp_decomp}
for an overview. To approximately decompose 
tensor $\scr T \in 
\RR^{I_1 \times ... \times I_N}$, 
the popular alternating
least squares (ALS) algorithm begins with 
randomly initialized factor
matrices $U_j$, $U_j \in \RR^{I_j \times R}$
for $1 \leq j \leq N$. We call the column count
$R$ the \textbf{rank} of the decomposition. 
Each round of ALS solves $N$ overdetermined least squares problems in sequence, each optimizing
a single factor matrix while holding the others constant. The $j$-th least squares
problem occurs in the update 
\[
U_j := \argmin_{X} \norm{U_{\neq j} \cdot X^\top - \textrm{mat}(\scr T, j)^\top}_F
\]
where $U_{\neq j} = U_N \odot 
... \odot U_{j+1} \odot U_{j-1} \odot ... \odot U_1$ 
is the Khatri-Rao product 
of all matrices excluding $U_j$ and
$\textrm{mat}(\cdot)$ denotes the mode-$j$ matricization of tensor $\scr T$. Here,  
we reverse the order of matrices in the
Khatri-Rao product to match the ordering of 
rows in the matricized tensor (see 
Appendix \ref{appendix:als_cp_decomp} for
an explicit formula for the matricization).
These problems are ideal candidates for randomized
sketching \cite{battaglino_practical, 
jin_faster_2020, larsen_practical_2022}, and applying 
the data structure in Theorem
\ref{thm:main_krp_res} gives us the \textbf{STS-CP} algorithm.
\begin{corollary}[STS-CP]
Suppose $\scr T$ is dense, and suppose we solve each least squares problem in ALS with a randomized sketching algorithm. A leverage score sampling approach as defined 
in section \ref{sec:preliminaries} guarantees that 
with $\tilde O(R / (\varepsilon \delta))$ samples per solve, the
residual of each sketched least squares problem is within 
$(1 + \varepsilon)$ of the optimal residual with probability
$(1 - \delta)$. The efficient sampler from Theorem 
\ref{thm:main_krp_res} brings the complexity of ALS to 
\[
\tilde O \paren{\frac{\noiter}{\varepsilon \delta} 
\cdot \sum_{j=1}^N \paren{N R^3 \log I_j + I_j R^2 }}
\]
where ``$\noiter$" is the number of ALS iterations, and with any term 
$\log I_j$ replaced by $\log R$ if $I_j < R$.
\label{cor:downsampled_cpals_complexity}
\end{corollary}
The proof appears in Appendix \appendixref{appendix:als_cp_decomp} and combines Theorem
\ref{thm:main_krp_res} with Theorem 
\ref{thm:lev_score_lowerbounds}. STS-CP also 
works for sparse tensors 
and likely provides a greater advantage here than the dense case, as sparse
tensors tend to have much larger mode size \cite{frosttdataset}. 
The complexity for sparse tensors
depends heavily on the sparsity structure and is difficult to predict.
Nevertheless, we expect a significant speedup
based on prior works that use sketching to accelerate CP decomposition 
\cite{cheng_spals_2016,larsen_practical_2022}. 

\section{Experiments}
Experiments were conducted on CPU nodes of NERSC Perlmutter, an HPE Cray 
EX supercomputer, 
and our code is available at \url{https://github.com/vbharadwaj-bk/fast_tensor_leverage.git}. 
On tensor decomposition experiments, we 
compare our algorithms against the random and hybrid versions of CP-ARLS-LEV
proposed by \citet{larsen_practical_2022}.
These algorithms outperform uniform sampling and row-norm-squared sampling, achieving excellent accuracy and runtime 
relative to exact ALS. In contrast to TNS-CP and the
Gaussian tensor network embedding proposed by Ma and Solomonik
(see Table 1), CP-ARLS-LEV is one of the few algorithms that can
practically decompose sparse tensors with mode sizes in the
millions. In the worst case, CP-ARLS-LEV requires 
$\tilde O(R^{N-1} / (\varepsilon \delta))$ samples per solve for an 
$N$-dimensional tensor to achieve solution guarantees like those in
Theorem \ref{thm:lev_score_lowerbounds}, compared to 
$\tilde O(R / (\varepsilon \delta))$ samples required by STS-CP.
Appendices \appendixref{appendix:efficient_parallel_impl}, 
\appendixref{appendix:sparse_cp}, and \appendixref{appendix:supp_results} 
provide configuration details and additional results.

\subsection{Runtime Benchmark}

\begin{figure}
\centering
\begin{minipage}[t]{.48\textwidth}
  \centering
    \includegraphics[scale=0.43]{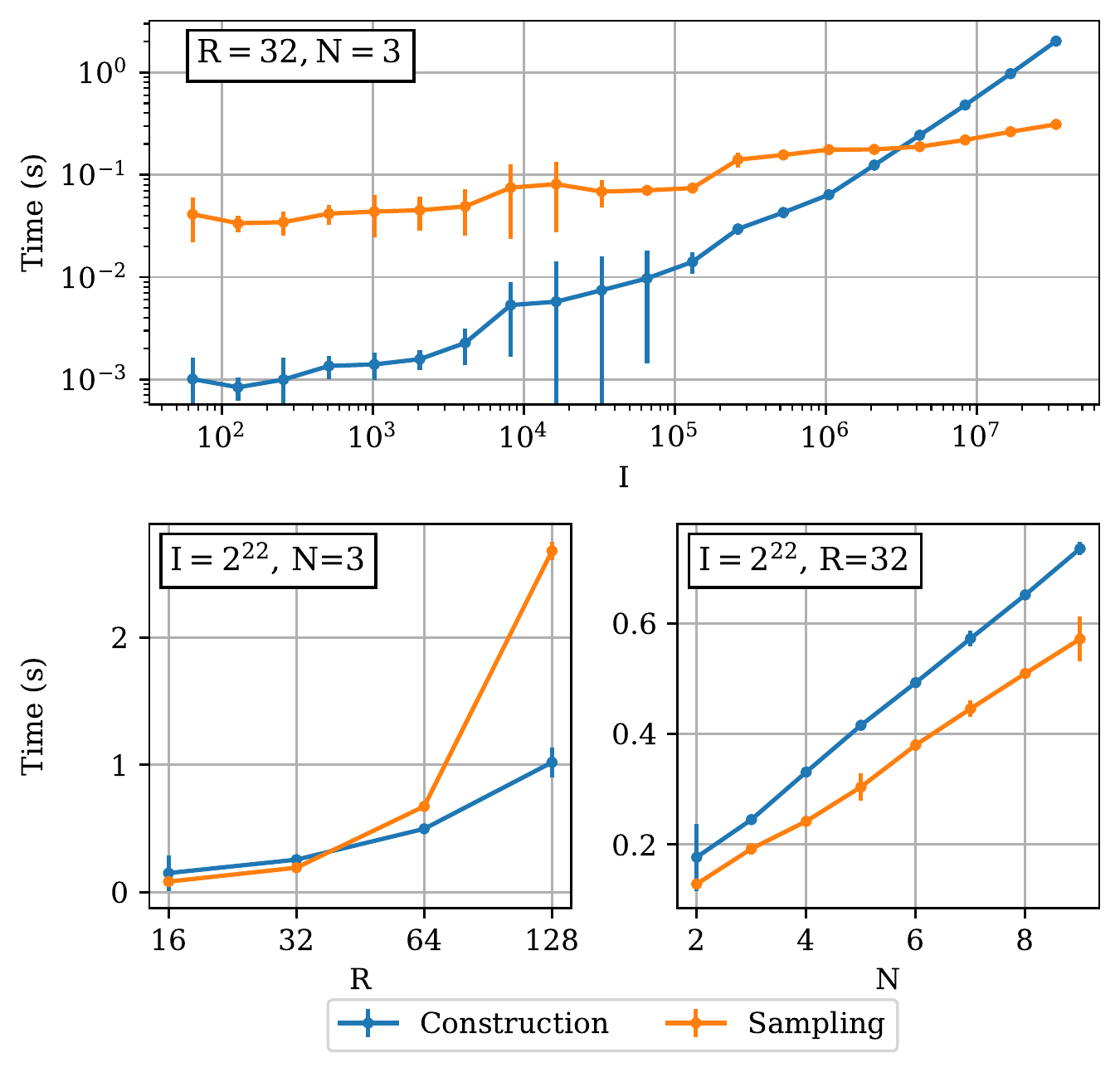}
        \captionof{figure}{Average time (5 trials) to construct our proposed sampler and draw $J=50,000$ samples from $U_1 \odot ... \odot U_N$, with
        $U_j \in \RR^{I \times R}\ \forall j$. Error bars indicate 3
        standard deviations.}
    \label{fig:runtime_benchmark} 
\end{minipage}
\hfill
\begin{minipage}[t]{.48\textwidth}
  \centering
  \includegraphics[scale=0.26]{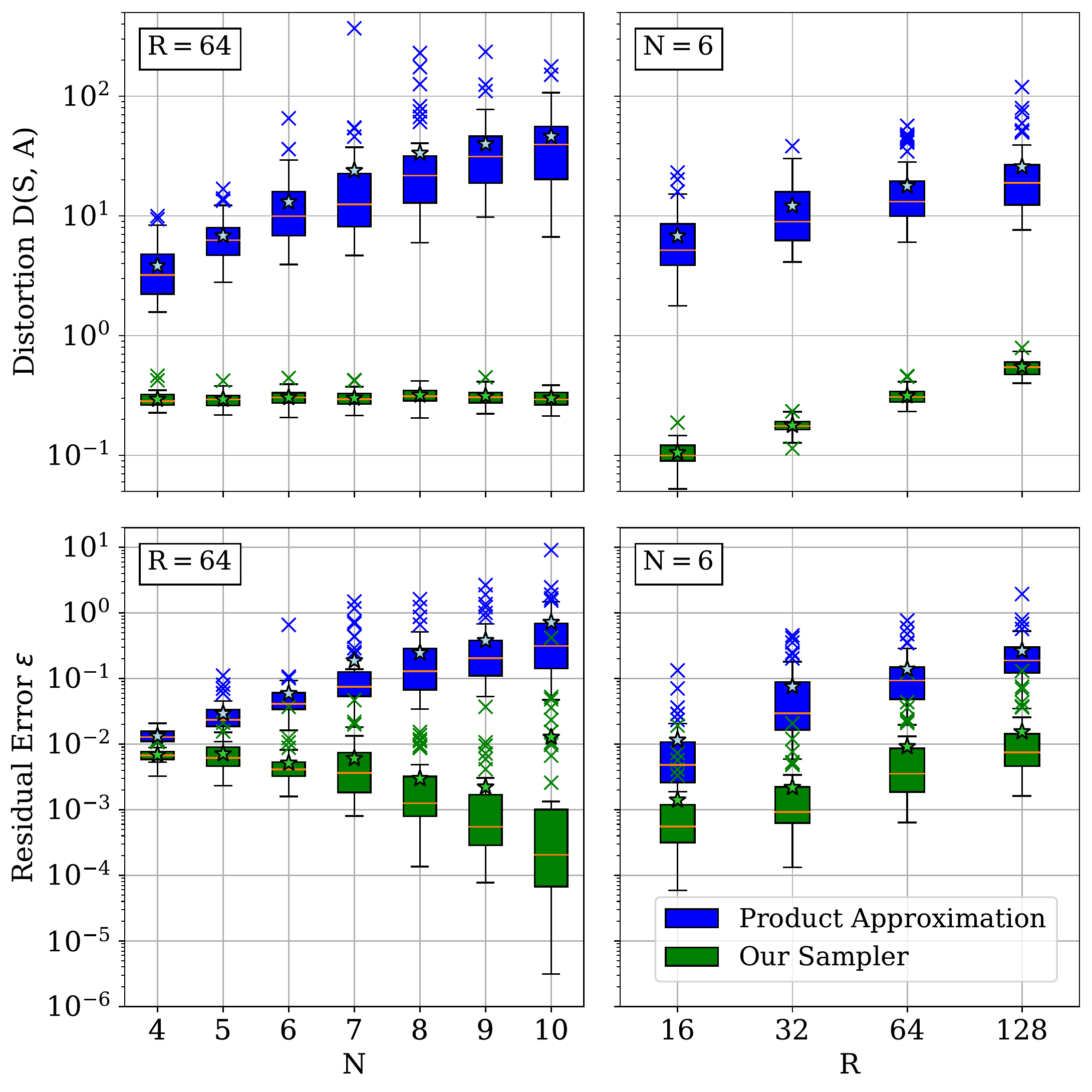}
        \captionof{figure}{Distortion and 
        residual error (50 
        trials) for varying $R$ and $N$ on 
        least squares, $I=2^{16}, J=5000$. ``X'' marks indicate
        outliers 1.5 times the interquartile range beyond the median,
        stars indicate means.
        }
    \label{fig:accuracy_bench} 
\end{minipage}
\end{figure}

Figure \ref{fig:runtime_benchmark} shows the time to construct our sampler and draw 50,000
samples from the Khatri-Rao product of i.i.d.\ Gaussian initialized factor matrices.
We quantify the runtime impacts of varying $N$, $R$, and $I$. The asymptotic
behavior in Theorem \ref{thm:main_krp_res} is reflected in our performance measurements, 
with the exception of the plot that varies $R$. Here, construction becomes 
disproportionately cheaper than sampling due to cache-efficient 
BLAS3 calls during construction. Even when the full Khatri-Rao product has 
$\approx 3.78 \times 10^{22}$ rows (for $I=2^{25}, N=3, R=32$), we require only $0.31$ seconds on average for sampling (top plot, rightmost points).

\subsection{Least Squares Accuracy Comparison}
We now test our sampler on least squares problems of the 
form $\min_x \norm{Ax-b}$, where $A = U_1 \odot ... \odot U_N$
with $U_j \in \RR^{I \times R}$ for all $j$. We initialize all matrices
$U_j$ entrywise i.i.d.\ from a standard normal distribution and randomly multiply
1\% of all entries by 10. We choose $b$ as a 
Kronecker product $c_1 \otimes ... \otimes c_N$, with each vector 
$c_j \in \RR^I$ also initialized entrywise from a Gaussian 
distribution. We assume this form for $b$ to tractably compute
the exact solution to the linear least squares problem and
evaluate the accuracy of our randomized methods. We  
\textbf{do not} give our algorithms access
to the Kronecker form of $b$; they are only permitted
on-demand, black-box access to its entries. 

For each problem instance, define the distortion of our sampling 
matrix $S$ with respect to the column space of $A$ as
\begin{equation}
D(S, A) = \kappa(SQ) - 1
\label{eq:distortion_defn}
\end{equation}
where $Q$ is an orthonormal basis for the column space of $A$ and
$\kappa(SQ)$ is the condition number of $SQ$.
A higher-quality sketch $S$ exhibits lower distortion, which 
quantifies the preservation of distances from the column space of $A$ 
to the column space of $SA$ \cite{murray2023randomized}. For details
about computing $\kappa(SQ)$ efficiently when $A$ is a Khatri-Rao product, 
see Appendix \ref{appendix:distortion}. Next, 
define $\varepsilon =\frac{\textrm{residual}_{\textrm{approx}}}{\textrm{residual}_{\textrm{opt}}}-1$,
where $\textrm{residual}_{\textrm{approx}}$ is the residual of the
randomized least squares algorithm. $\varepsilon$ is nonnegative 
and (similar to its role in Theorem \ref{thm:lev_score_lowerbounds}) quantifies the quality of the randomized algorithm's solution. 

For varying $N$ and $R$, Figure \ref{fig:accuracy_bench} shows 
the average values of $D$ and $\varepsilon$ achieved by
our algorithm against the leverage product approximation used by 
\citet{larsen_practical_2022}. Our sampler exhibits nearly constant 
distortion $D$ for fixed rank $R$ and varying $N$, and it achieves 
$\varepsilon \approx 10^{-2}$ even when $N = 9$. 
The product approximation increases both the distortion and
residual error by at least an order of magnitude.

\subsection{Sparse Tensor Decomposition}
\label{sec:sparse_tensor_cp_experiments}
\begin{figure*}
    \centering
    \includegraphics[scale=0.42]{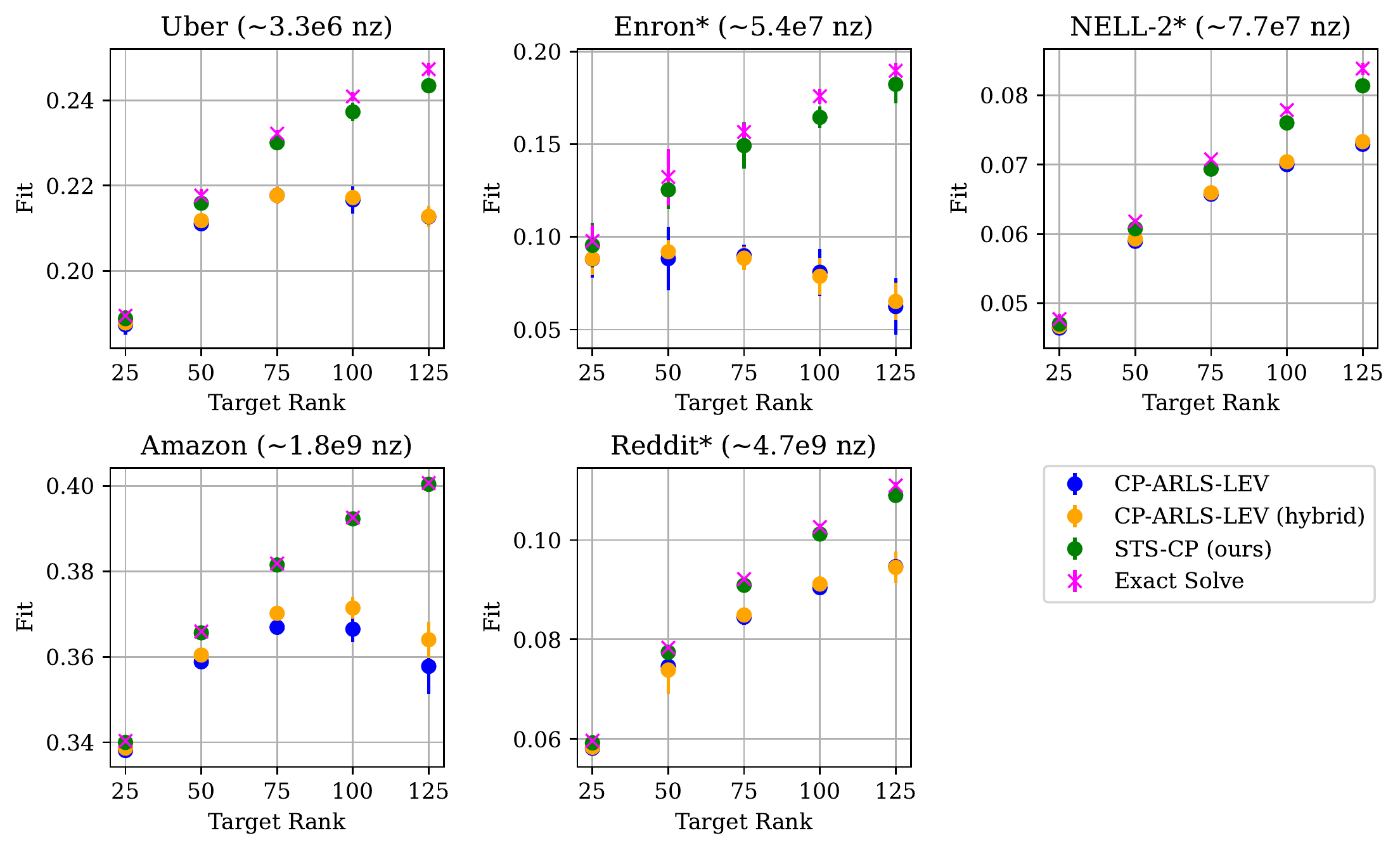}
    \caption{Average fits (8 trials) achieved by randomized ($J=2^{16}$) and exact ALS for
    sparse tensor CP decomposition. Error bars 
    indicate 3 standard deviations. See Appendix \appendixref{appendix:sparse_cp} for details.}
    \label{fig:accuracy_comparison} 
\end{figure*}

We next apply STS-CP to decompose several large sparse tensors from the
FROSTT collection \cite{frosttdataset} (see Appendix \appendixref{appendix:sparse_cp} for more
details on the experimental configuration). 
Our accuracy metric is the tensor fit. Letting $\tilde{\scr T}$ be our low-rank CP approximation, the fit
with respect to ground-truth tensor $\scr T$ is 
$\textrm{fit}(\tilde{\scr T}, \scr T) = 1 - \norm{\tilde{\scr T} - \scr T}_F / \norm{\scr T}_F$. 

Table \ref{tab:baseline-comparison} in Appendix
\ref{appendix:cp_baseline_comparison} compares
the runtime per round of our algorithm against the
Tensorly Python package \cite{tensorly} 
and Matlab Tensor Toolbox \cite{ttoolbox_sparse}, 
with dramatic speedup over both. As 
Figure \ref{fig:accuracy_comparison} shows,
the fit achieved by CP-ARLS-LEV compared to STS-CP degrades as the rank increases for fixed sample
count. By contrast, STS-CP improves the fit consistently, with a significant improvement at rank 125 over CP-ARLS-LEV. Timings
for both algorithms are available in Appendix 
\appendixref{appendix:sts-cp-speedup}. 
Figure \ref{fig:exact_solve_comparison_sparse} explains the higher fit achieved by our sampler on the Uber and Amazon tensors. In the first 10 rounds of 
ALS, we compute the exact solution to each least squares problem
before updating the factor matrix with a randomized algorithm's solution. Figure \ref{fig:exact_solve_comparison_sparse} 
plots $\varepsilon$ as ALS progresses for hybrid CP-ARLS-LEV and STS-CP. 
The latter consistently achieves lower residual per solve. 
We further observe that 
CP-ARLS-LEV exhibits an oscillating error pattern with period matching the
number of modes
$N$.

To assess the tradeoff between sampling time and accuracy, we
compare the fit as a function of ALS update time for STS-CP and random CP-ARLS-LEV
in Figure \ref{fig:reddit_fit_function_time} (time to compute the fit excluded). On the Reddit tensor 
with $R=100$, we compared CP-ARLS-LEV with $J=2^{16}$ against CP-ARLS-LEV with progressively
larger sample count. Even with $2^{18}$ samples per randomized least squares solve, 
CP-ARLS-LEV cannot achieve the maximum fit of STS-CP. Furthermore, 
STS-CP makes progress more quickly than CP-ARLS-LEV. See
Appendix \appendixref{appendix:extra_fit_times} for similar plots for other datasets. 

\begin{figure}
\centering
\begin{minipage}[t]{.48\textwidth}
  \centering
  \includegraphics[scale=0.40]{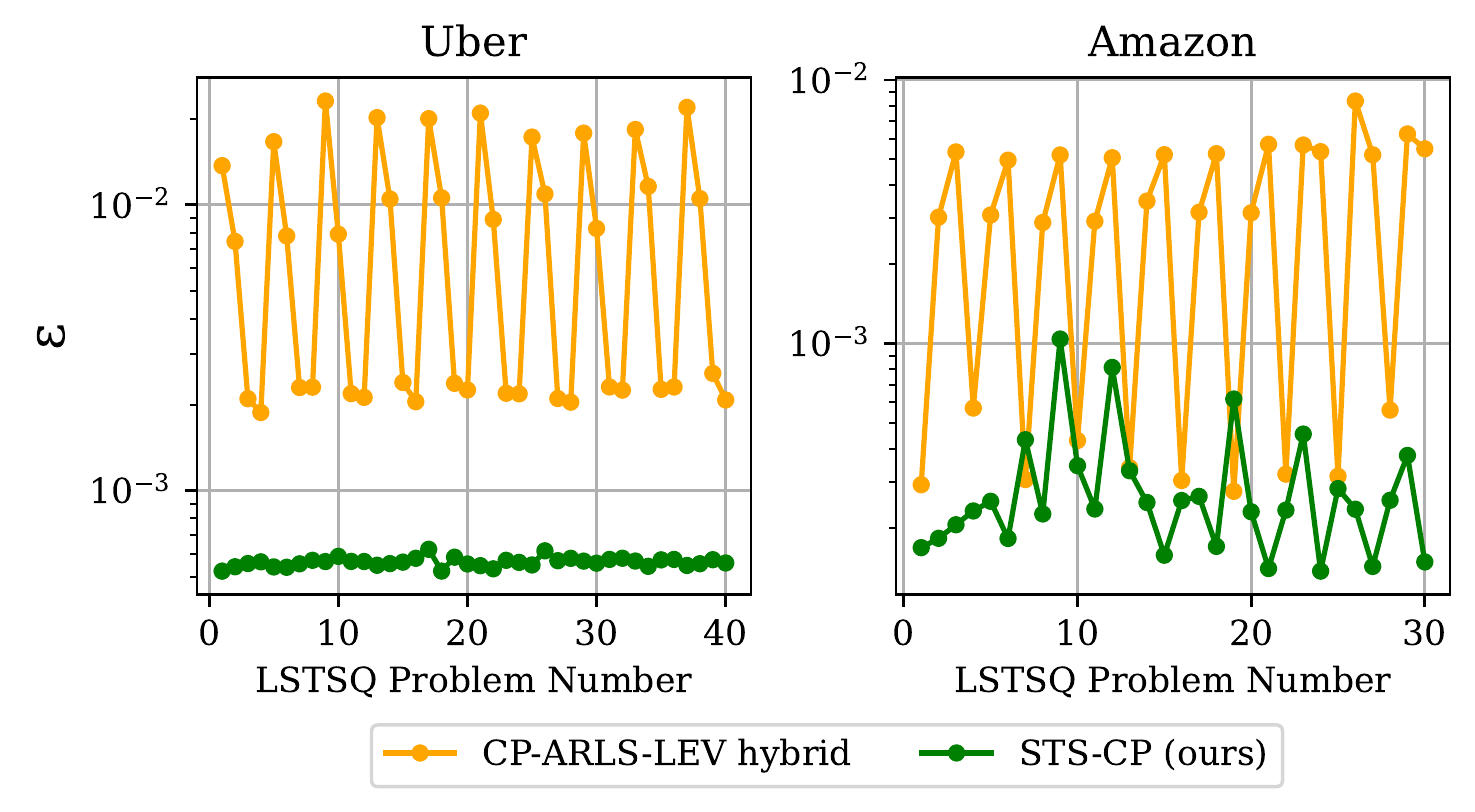}
    \captionof{figure}{Average $\varepsilon$ (5 runs) 
    for randomized least squares 
    solves in 10 ALS rounds, $R=50$.}
    \label{fig:exact_solve_comparison_sparse}
\end{minipage}
\hfill
\begin{minipage}[t]{.48\textwidth}
  \centering
    \includegraphics[scale=0.35]{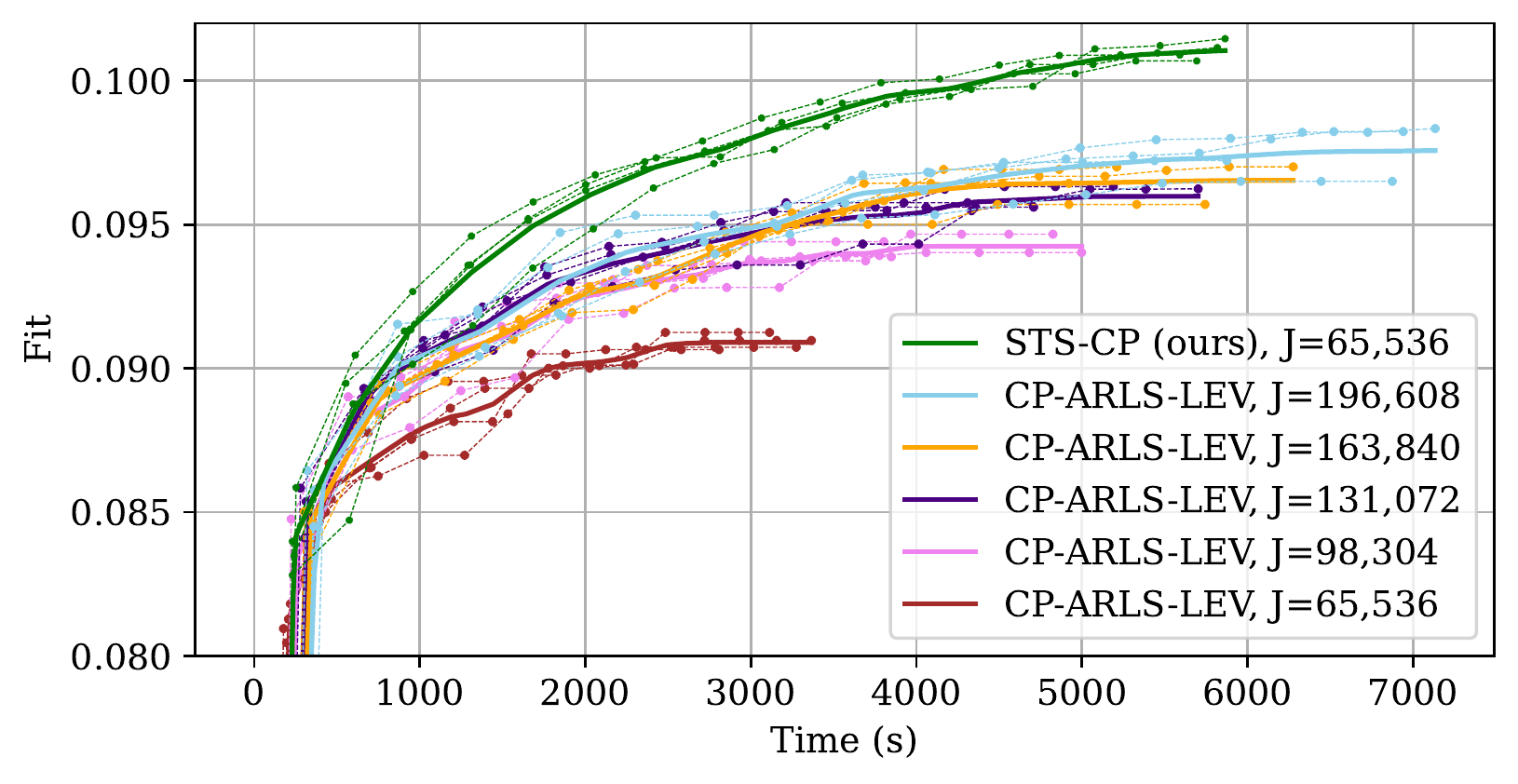}
    \captionof{figure}{Fit vs.\ time, Reddit tensor, $R=100$. Thick lines are averages 4 trial 
    interpolations.} 
    \label{fig:reddit_fit_function_time}
\end{minipage}
\end{figure}

\section{Discussion and Future Work}
Our method for exact Khatri-Rao leverage score sampling enjoys 
strong theoretical guarantees and practical performance benefits. 
Especially for massive tensors such as Amazon and
Reddit, our randomized algorithm's guarantees translate to faster
progress to solution and higher final accuracies. The segment 
tree approach described here can be applied to sample from tensor
networks besides the Khatri-Rao product. In particular, modifications to Lemma
\ref{lemma:efficient_row_sampler} permit efficient leverage sampling from a contraction of 3D tensor
cores in ALS tensor train decomposition. We leave the generalization of 
our fast sampling technique as future work.
\clearpage

\section*{Acknowledgements, Funding, and Disclaimers}
We thank the referees for valuable feedback which helped improve the paper.

V. Bharadwaj was supported by the U.S. Department 
of Energy, Office of Science, Office of Advanced Scientific 
Computing Research, Department of Energy Computational Science 
Graduate Fellowship under Award Number DE-SC0022158. O. A. Malik and
A. Bulu\c{c} were supported by the Office of Science of the DOE
under Award Number DE-AC02-05CH11231. L. Grigori was supported 
by the European Research Council (ERC) under the European
Union’s Horizon 2020 research and innovation program through grant agreement 810367. R. Murray was supported by Laboratory 
Directed Research and Development (LDRD) funding from Berkeley Lab, provided by the Director, Office of Science, of the U.S. DOE 
under Contract No. DE-AC02-05CH11231. R. Murray was also funded by an NSF Collaborative Research Framework under NSF Grant Nos. 2004235 and 2004763. This research used resources 
of the National Energy Research Scientific Computing Center, a DOE Office 
of Science User Facility, under Contract No. DE-AC02-05CH11231 using 
NERSC award ASCR-ERCAP0024170. 

This report was prepared as an account of work sponsored by an agency of the
United States Government. Neither the United States Government nor any agency thereof, nor
any of their employees, makes any warranty, express or implied, or assumes any legal liability
or responsibility for the accuracy, completeness, or usefulness of any information, apparatus,
product, or process disclosed, or represents that its use would not infringe privately owned
rights. Reference herein to any specific commercial product, process, or service by trade name,
trademark, manufacturer, or otherwise does not necessarily constitute or imply its
endorsement, recommendation, or favoring by the United States Government or any agency
thereof. The views and opinions of authors expressed herein do not necessarily state or reflect
those of the United States Government or any agency thereof.

\bibliographystyle{plainnat}
\bibliography{references}
\clearpage


\appendix
\section{Appendix}
\subsection{Details about Table \ref{tab:cp-complexity-comparison}}
\label{appendix:intro}
CP-ALS \cite{kolda_tensor_overview}
is the standard, non-randomized
alternating least squares method given by
Algorithm \ref{alg:als_algorithm} in
Appendix \appendixref{appendix:als_cp_decomp}.
The least squares problems in the algorithm
are solved by exact methods. 
CP-ARLS-LEV is the algorithm proposed by \citet{larsen_practical_2022} that samples rows from the 
Khatri-Rao product according to a product distribution of
leverage scores on each factor matrix. The per-iteration runtimes  
for both algorithms are re-derived in 
Appendix C.3 of the work by \citet{malik_efficient_2022} from their 
original sources. \citet{malik_efficient_2022} 
proposed the CP-ALS-ES algorithm (not listed in the table),
which is superseded by the TNS-CP algorithm \cite{malik_tns_cp}. We report the complexity
from Table 1 of the latter work. The algorithm by 
\citet{ma_cost_efficient_embedding} is based on a general method to sketch tensor
networks. Our reported complexity is listed in Table 1 for Algorithm 1 in their work. 

Table \ref{tab:cp-complexity-comparison} does not list the one-time initialization costs for any of the
methods. All methods require at least $O(NIR)$ time to randomly initialize factor matrices, and CP-ALS
requires no further setup. CP-ARLS-LEV, TNS-CP, and STS-CP all require $O(NIR^2)$ initialization time. 
CP-ARLS-LEV uses the initialization phase to compute the initial leverage scores of all factor matrices. 
TNS-CP uses the initialization step to compute and cache Gram matrices of all factors $U_j$.
STS-CP must build the efficient sampling data structure described in Theorem \ref{thm:main_krp_res}. The
algorithm from Ma and Solomonik requires an initialization cost of $O(I^N m)$, where $m$ is
a sketch size parameter that is $O(NR / \varepsilon^2)$ to achieve the $(\varepsilon, \delta)$ accuracy guarantee for each
least squares solve.

\subsection{Definitions of Matrix Products}
\label{appendix:notation}

Table \ref{tab:matrix-product-definitions} defines the standard
matrix product $\cdot$, Hadamard product $\circledast$, Kronecker
product $\otimes$, and Khatri-Rao product $\odot$, as well
as the dimensions of their operands.

\begin{table}[htb]
\caption{Matrix product definitions.} 
\label{tab:matrix-product-definitions}
\begin{center}
\begin{small}
\begin{sc}
	\begin{tabular}{llllll}  
		\toprule
		Operation & Size of $A$ & Size of $B$ & Size of $C$ & Definition\\
		\midrule
        $C=A \cdot B$ & $(m, k)$ & $(k, n)$ &
        $(m, n)$ &
        $C[i,j] = \sum_{a=1}^k A\br{i, a} B\br{a, j}$ 
        \\
        $C=A \circledast B$ & $(m, n)$ & $(m, n)$ &
        $(m, n)$ 
        & $C \br{i, j} = A\br{i, j} B \br{i, j}$ 
        \\ 
        $C=A \otimes B$ & $(m_1, n_1)$ & $(m_2, n_2)$ & $(m_2 m_1, n_2 n_1)$ & $C\br{(i_2, i_1), (j_2, j_1)} = A\br{i_1, j_1} B \br{i_2, j_2}$ \\ 
        $C = A \odot B$ & $(m_1, n)$ & $(m_2, n)$ & $(m_2 m_1, n)$ 
        &  $C\br{(i_2, i_1), j} = A\br{i_1, j} B \br{i_2, j}$ 
        \\
		\bottomrule
	\end{tabular} 
\end{sc}
\end{small}
\end{center}
\end{table}

\subsection{Further Comparison to Prior Work}
\label{appendix:wz_comparison}
In this section, we provide a more detailed comparison of our sampling algorithm with the one 
proposed by \citet{woodruff_zandieh}. Their work introduces a ridge 
leverage-score sampling algorithm for Khatri-Rao products with the attractive property that the 
sketch can be formed in input-sparsity time. For constant
failure probability $\delta$, the runtime to produce
a $(1 \pm \epsilon)$ $\ell_2$-subspace embedding for $A = U_1 \odot ... \odot U_N$ is given in Appendix 
B of their work (proof of Theorem 2.7). Adapted to our notation, \textbf{their runtime} is 
\[
O \paren{
\log^4 R \log N \sum_{i=1}^N \textrm{nnz}(U_i) +
\frac{N^7 s_\lambda^2 R}{\varepsilon^4} \log^5 R \log N 
}
\]
where $s_{\lambda} = \sum_{i=1}^R \frac{\lambda_i}{\lambda_i + \lambda}$, $\lambda_1,...,\lambda_R$ 
are the eigenvalues of the Gram matrix $G$ of matrix $A$, and $\lambda \geq 0$ is 
a regularization parameter. For comparison, \textbf{our runtime} for constant failure probability $\delta$ is 
\[
O \paren{
R \sum_{i=1}^N \textrm{nnz}(U_i) +
\frac{R^3}{\varepsilon} \log \paren{\prod_{i=1}^N I_i} \log R}.
\]
Woodruff and Zandieh's method provides a significant advantage for large column count $R$ or 
high regularization parameter $\lambda$. As a result, it is well-suited to the problem of 
regularized low-rank approximation when the column count $R$ is given by the number of data points
in a dataset. On the other hand, the algorithm has poor dependence on 
the matrix count $N$ and error parameter $\varepsilon$. For tensor decomposition,
$R$ is typically no larger than a few hundred, while high accuracy ($\epsilon \approx 10^{-3}$)
is required for certain tensors to achieve a fit competitive with non-randomized methods (see
section \ref{sec:sparse_tensor_cp_experiments}, Figures \ref{fig:accuracy_comparison} 
and \ref{fig:exact_solve_comparison_sparse}). 
When $\lambda$ is small, we have $s_\lambda \approx R$. 
Here, Woodruff and Zandieh's runtime has an $O(R^3)$ dependence
similar to ours. When $R \leq \log^4 R \log N$, our sampler 
has faster construction time as well. 

Finally, we note that our sampling data structure can be constructed using
highly cache-efficient, parallelizable symmetric rank-$R$ updates (BLAS3 operation \verb|dSYRK|).
As a result, the quadratic dependence on $R$ in our algorithm can be mitigated by dense linear 
algebra accelerators, such as GPUs or TPUs. 

\subsection{Proof of Theorem \ref{thm:malik2022}}
\label{appendix:malik_reproof}
Theorem \ref{thm:malik2022} appeared in a modified form as Lemma 10 in the work by \citet{malik_efficient_2022}. 
This original version used the definition 
$\tilde G_{>k} = \Phi \circledast \startimes_{a=k+1}^N G_k$ in place of $G_{>k}$
defined in Equation \eqref{eq:G_geq_k_defn}, where $\Phi$ was a sketched approximation of $G^+$.
\citet{woodruff_zandieh} exhibit a version of the theorem with similar modifications. We prove the version stated in our work below. 
\begin{proof}[Proof of Theorem \ref{thm:malik2022}]
We rely on the assumption that the Khatri-Rao product $A$ is
a nonzero matrix (but it may be rank-deficient). 
We begin by simplifying the expression for the leverage score
of a row of $A$ corresponding to multi-index $(i_1, ..., i_N)$.
Beginning with Equation \eqref{eq:krp_leverage}, we derive
\begin{equation}
\begin{aligned}
\ell_{i_1, ..., i_{N}} 
&=A \br{(i_1, ..., i_{N}), :} G^+ A \br{(i_1, ..., i_{N}), :}^\top \\
&=\gen{A \br{(i_1, ..., i_{N}), :}^\top A \br{(i_1, ..., i_{N}), :}, G^+} \\
&=\gen{\paren{ \startimes_{a=1}^{N} U_a\br{i_a, :}}^\top 
\paren{ \startimes_{a=1}^{N} U_a\br{i_a, :}}, G^+} \\
&=\gen{\startimes_{a=1}^N U_a\br{i_a, :}^\top U_a \br{i_a, :}, G^+} \\
&=\gen{\startimes_{a=1}^{k-1} U_a\br{i_a, :}^\top U_a \br{i_a, :},
U_k \br{i_k, :}^\top U_k \br{i_k, :} \circledast
\startimes_{a=k+1}^{N} U_a\br{i_a, :}^\top U_a \br{i_a, :}, G^+} \\
&=\gen{\startimes_{a=1}^{k-1} U_a\br{i_a, :}^\top U_a \br{i_a, :}, 
U_k \br{i_k, :}^\top U_k \br{i_k, :}, G^+
\circledast \startimes_{a=k+1}^{N} U_a\br{i_a, :}^\top U_a \br{i_a, :}}.
\label{eq:malik_deriv_1}
\end{aligned}
\end{equation}
We proceed to the main proof of the theorem. To compute $p(\hat s_k = s_k\ \vert\ \hat s_{<k} = s_{<k})$, we marginalize over
random variables $\hat s_{k+1} ... \hat s_N$. Recalling the definition of 
$h_{<k}$ from Equation \eqref{eq:h_leq_k_defn}, we have 
\begin{equation}
\begin{aligned}
p(\hat s_k = s_k\ \vert\ \hat s_{<k} = s_{<k})
&\propto \sum_{i_{k+1}, ..., i_{N}} 
p\paren{(\hat s_{<k} = s_{<k}) \wedge (\hat s_k = s_k) \wedge
\bigwedge_{u=k+1}^{N} (\hat s_u = i_u)} \\
&\propto \sum_{i_{k+1}, ..., i_{N}} 
\ell_{s_1, ..., s_k, i_{k+1}, ..., i_{N}}. \\
\end{aligned}
\end{equation}
The first line above follows by marginalizing over
$\hat s_{k+1}, ..., \hat s_N$. The second line follows
because the joint random variable $(\hat s_1, ..., \hat s_N)$ follows
the distribution of statistical leverage scores on the rows
of $A$. We now plug in Equation \eqref{eq:malik_deriv_1} to get
\begin{equation}
\begin{aligned}
&\sum_{i_{k+1}, ..., i_{N}} 
\ell_{s_1, ..., s_k, i_{k+1}, ..., i_{N}} \\
&= \sum_{i_{k+1}, ..., i_{N}} \gen{\startimes_{a=1}^{k-1} U_a\br{s_a, :}^\top U_a \br{s_a, :}, 
U_k \br{s_k, :}^\top U_k \br{s_k, :}, G^+
\circledast \startimes_{a=k+1}^{N} U_a\br{i_a, :}^\top U_a \br{i_a, :}} \\
&= \sum_{i_{k+1}, ..., i_{N}} \gen{h_{<k} h_{<k}^\top, 
U_k \br{s_k, :}^\top U_k \br{s_k, :}, G^+
\circledast \startimes_{a=k+1}^{N} U_a\br{i_a, :}^\top U_a \br{i_a, :}} \\
&= \gen{h_{<k} h_{<k}^\top, 
U_k \br{s_k, :}^\top U_k \br{s_k, :}, G^+
\circledast \startimes_{a=k+1}^{N} 
\sum_{i_a=1}^{I_a} U_a\br{i_a, :}^\top U_a \br{i_a, :}} \\
&= \gen{h_{<k} h_{<k}^\top, 
U_k \br{s_k, :}^\top U_k \br{s_k, :}, G^+
\circledast \startimes_{a=k+1}^{N} G_a} \\
&= \gen{h_{<k} h_{<k}^\top, 
U_k \br{s_k, :}^\top U_k \br{s_k, :}, G_{>k}}.
\label{eq:malik_thm_deriv2}
\end{aligned}
\end{equation}
We now compute the normalization constant $C$ for the distribution by summing the last line of
Equation \eqref{eq:malik_thm_deriv2} over all possible values for $\hat s_k$:
\begin{equation}
\begin{aligned}
C &= \sum_{s_k=1}^{I_k} \gen{h_{<k} h_{<k}^\top, 
U_k \br{s_k, :}^\top U_k \br{s_k, :}, G_{>k}} \\
&=  \gen{h_{<k} h_{<k}^\top, \sum_{s_k=1}^{I_k} U_k \br{s_k, :}^\top U_k \br{s_k, :}, G_{>k}} \\
&= \gen{h_{<k} h_{<k}^\top, G_k, G_{>k}}.
\end{aligned}
\end{equation}
For $k = 1$, we have $h_{<k} = \br{1, ..., 1}^\top$, so
$C = \gen{G_k, G_{>k}}$. Then $C$ is the sum of all leverage scores, 
which is known to be the rank of $A$ \cite{woodruff2014sketching}. 
Since $A$ was assumed nonzero, $C \neq 0$. For $k > 1$, assume that the
conditioning event $\hat s_{<k} = s_{<k}$ occurs with nonzero probability. This is a reasonable
assumption, since our sampling algorithm will never 
select prior values 
$\hat s_1, ..., \hat s_{k-1}$ that
have 0 probability of occurrence. Let $\tilde C$
be the normalization constant for the conditional 
distribution on $\hat s_{k-1}$. Then we have 
\begin{equation}
\begin{aligned}
0
&<p(\hat s_{k-1} = s_{k-1}\ \vert\ \hat s_{<k-1} = s_{<k-1}) \\
&= \tilde C^{-1} \gen{h_{<k-1} h_{<k-1}^\top, 
U_{k-1}\br{s_{k-1}, :}^\top U_{k-1}\br{s_{k-1}, :}, G_{>k-1}} \\
&= \tilde C^{-1} \gen{h_{<k} h_{<k}^\top, G_{>k-1}} \\
&= \tilde C^{-1} \gen{h_{<k} h_{<k}^\top, G_k \circledast G_{>k}} \\
&= \tilde C^{-1} \gen{h_{<k} h_{<k}^\top, G_k, G_{>k}} \\
&= \tilde C^{-1} C
\end{aligned}
\end{equation}
Since $\tilde C > 0$, we must have $C > 0$.
\end{proof}

\subsection{Proof of Lemma \ref{lemma:efficient_row_sampler}}
\label{appendix:efficient_lemma_proof}
\begin{proof}[\unskip\nopunct] We detail the construction procedure, 
sampling procedure, and correctness of our proposed data structure. Recall
that $T_{I, F}$ denotes the collection of nodes in a full, complete
binary tree with $\ceil{I / F}$ leaves. Each leaf $v \in T_{I, F}$ holds a 
segment $S(v) = \set{S_0(v)..S_1(v)} \subseteq \set{1..I}$, with $\abs{S(v)} \leq F$ and
$S(u) \cap S(v) = \varnothing$ for distinct leaves $u, v$. For each
internal node $v$, $S(v) = S(L(v)) \cup S(R(v))$, where $L(v)$ and
$R(v)$ denote the left and right children of node $v$. The root node $r$
satisfies $S(r) = \set{1..I}$.

\textbf{Construction:} Algorithm \ref{alg:row_sampler_construction} gives the procedure
to build the data structure. We initialize a segment tree $T_{I, F}$ and 
compute $G^v$ for all leaf nodes $v \in T_{I, F}$ as a sum of outer products of
rows from $U$ (lines 1-3). Starting at the level above the leaves, we then compute
$G^v$ for each internal node as the sum of $G^{L(v)}$ and $G^{R(v)}$, the partial Gram matrices of its two children. Runtime $O(I R^2)$ 
is required to compute $I$ outer products across all 
iterations of the loop on line 3. Our  
segment tree has $\ceil{I / F} - 1$ internal nodes, and the addition in line 6
contributes runtime $O(R^2)$ for each internal node. This adds complexity 
$O(R^2(\ceil{I / F} - 1)) \leq O(I R^2)$, for total construction time $O(I R^2)$.

To analyze the space complexity, observe that we store a matrix 
$G^v \in \RR^{R \times R}$ at all $2 \ceil{I / F} - 1$ nodes of the segment 
tree, for asymptotic space usage $O(\ceil{I/F} R^2)$. We can cut the space usage in half 
by only storing $G^v$ when $v$ is either the root or a left child in our tree, 
since the sampling procedure in 
Algorithm \ref{alg:row_sampler_sampling} only
accesses the partial Gram matrix stored by left children. 
We can cut the space 
usage in half again by only storing the upper triangle 
of each symmetric matrix $G^v$. Finally, in the special case that
$I < F$, the segment tree has depth 1 and the initial
binary search can be eliminated entirely. As a result, 
the data structure has $O(1)$ space
overhead, since we can avoid storing any partial Gram matrices $G^v$. This proves the
complexity claims in point 1 of Lemma \ref{lemma:efficient_row_sampler}.

\begin{algorithm}
   \caption{BuildSampler($U \in \RR^{I \times R}$, $F$, $Y$)}
   \label{alg:row_sampler_construction}
\begin{algorithmic}[1]
        \STATE Build tree $T_{I, F}$ with depth $d = \ceil{\log \ceil{I / F}}$
        \FOR{$v \in \textrm{leaves}(T_{I, F})$}
            \STATE $G^v := \sum_{i \in S(v)} U\br{i, :}^\top U\br{i, :}$
        \ENDFOR
        \FOR{$u=d-2...0$}
            \FOR{$v \in \textrm{level}(T_{I, F}, u)$}
                \STATE $G^{v} := G^{L(v)} + G^{R(v)}$
            \ENDFOR
        \ENDFOR        
\end{algorithmic}
\end{algorithm}

\textbf{Sampling: } Algorithm \ref{alg:row_sampler_sampling} gives the procedure
to draw a sample from our proposed data structure. It is easy to verify that
the normalization constant $C$ for $q_{h, U, Y}$ is 
$\gen{h h^\top, G^{\textrm{root}(T_{I, F})}, Y}$, since $G^{\textrm{root}(T_{I,F})} 
= U^\top U$. Lines 8 and 9 initialize a pair of 
templated procedures $\tilde m$ and $\tilde q$, each of which accepts a node
from the segment tree. The former is used to compute the 
branching threshold at
each internal node, and the latter returns the probability vector
$q_{h, U, Y}\br{S_0(v):S_1(v)}$ for the segment $\set{S_0(v)..S_1(v)}$ maintained by
a leaf node. To see this last fact, observe for $i \in \br{I}$ that 
\begin{equation}
\begin{aligned}
    &\tilde q(v)\br{i - S_0(v)} \\
    &= C^{-1} U\br{i, :} \cdot (h h^\top \circledast Y) \cdot U\br{i, :}^\top \\
    &= C^{-1} \gen{h h^\top, U \br{i, :}^\top U \br{i, :}, Y} \\ 
    &= q_{h, U, Y} \br{i}.
    \label{eq:tilde_q_correctness}
\end{aligned}
\end{equation}
The loop on line 12 performs the binary search using the two
templated procedures. Line 18 uses the procedure $\tilde q$ to scan through at most 
$F$ bin endpoints after the binary search finishes early.

The depth of segment tree $T_{I, F}$ is $\log \ceil{I / F}$. As a
result, the runtime of the sampling procedure is dominated by 
$\log \ceil{I / F}$ evaluations of $\tilde m$
and a single evaluation of $\tilde q$ during the binary search. Each execution of 
procedure $\tilde m$ 
requires time $O(R^2)$, relying on the partial Gram matrices $G^v$ 
computed during the construction phase. When $Y$ is a general 
p.s.d. matrix, the runtime of $\tilde q$ is $O(FR^2)$.
This complexity is dominated by the matrix multiplication
$W \cdot (h h^\top \circledast Y)$ on line 5. 
In this case, the runtime of the ``RowSampler" procedure to draw one sample is 
$O(R^2 \log \ceil{I / F} + F R^2)$, satisfying the complexity claims in point 2 of the
lemma.

Now supppose $Y$ is a rank-1 matrix with
$Y = u u^\top$ for some vector $u$.
We have $h h^\top \circledast Y 
= (h \circledast u) 
(h \circledast u)^\top$. This gives 
\[
\tilde q_p(h, C, v) = 
\textrm{diag} (W \cdot (h h^\top
\circledast u u^\top) 
\cdot W) = (W \cdot (h \circledast u))^2
\]
where the square is elementwise. The runtime of the procedure $\tilde q$ is now dominated 
by a matrix-vector multiplication that costs time $O(FR)$. In this case, we have 
per-sample complexity $O(R^2 \log \ceil{I / F} + F R)$, matching the complexity claim in 
point 3
of the lemma.
\begin{algorithm}
   \caption{Row Sampling Procedure}
   \label{alg:row_sampler_sampling}
\begin{algorithmic}[1]
        \REQUIRE Matrices $U, Y$ saved from construction, partial Gram matrices $\set{G^v\ \vert\ v \in T_{I, F}}$.
        \STATE \textbf{procedure} $m_p(h, C, v)$
        \begin{ALC@g}
            \STATE \textbf{return} $C^{-1} \gen{h h^\top, G^{v}, Y}$ 
        \end{ALC@g} 

        \STATE \textbf{procedure} $q_p(h, C, v)$
        \begin{ALC@g}
            \STATE $W := U\br{S(v), :}$
            \STATE \textbf{return} $C^{-1} \textrm{diag}(W \cdot (h h^\top \circledast Y) \cdot W^\top)$
        \end{ALC@g}  
        \STATE \textbf{procedure} RowSample($h$) 
        \begin{ALC@g}
            \STATE $C := \gen{h h^\top, G^{\textrm{root}(T_{I, F})}, Y}$
            \STATE $\tilde m(\cdot) := m_p(h, C, \cdot)$ 
            \STATE $\tilde q(\cdot) := q_p(h, C, \cdot)$ 
            \STATE $c := \rt{T_{I,F}}, \textrm{low}=0.0, \textrm{high}=1.0$
            \STATE Sample $d \sim \textrm{Uniform}(0.0, 1.0)$
            \WHILE {$c \notin \textrm{leaves}(T_{I,F})$} \label{rss_alg:binsearch}
                \STATE $\textrm{cutoff} := \textrm{low} + \tilde m(L(c))$
                \IF{$\textrm{cutoff} \geq d$}
                    \STATE $c := L(c), \textrm{high} := \textrm{cutoff}$
                \ELSE
                    \STATE $c := R(c), \textrm{low} := \textrm{cutoff}$ 
                \ENDIF
            \ENDWHILE
        
            \STATE \textbf{return} $S_0(v) + \argmin_{i \geq 0} 
                        \paren{\textrm{low} + \sum_{j=1}^i \tilde q(c) \br{j} < d}$
        \end{ALC@g}  
\end{algorithmic}
\end{algorithm}

\textbf{Correctness:} Recall that the inversion sampling procedure partitions the
interval $\br{0, 1}$ into $I$ bins, the $i$-th bin having width $q_{h, U, Y}\br{i}$. 
The goal of our procedure is to find the bin that contains the uniform random draw $d$.
Since procedure $\tilde m$ correctly returns the branching threshold
(up to the offset ``low'') given by Equation \eqref{eq:partial_gram}, the loop on line 12 correctly
implements a binary search on the list of bin endpoints specified by the vector $q_{h, U, Y}$.
At the end of the loop, $c$ is a leaf node that maintains a
collection $S(c)$ of bins, one
of which contains the random draw $d$. Since the procedure $\tilde q$ correctly returns
probabilities $q_{h, U, Y}\br{i}$ for $i \in S(c)$ for leaf
node $c$, (see 
Equation \eqref{eq:tilde_q_correctness}), line 18
finds the bin that contains 
the random draw $d$. The correctness of the procedure 
follows from the correctness of inversion
sampling \cite{saad_optimal_2020}.
\end{proof}

\subsection{Cohesive Proof of Theorem \ref{thm:main_krp_res}}
\label{appendix:cohesive_krp_proof}
\begin{proof}[\unskip\nopunct] 
In this proof, we fully explain Algorithms 
\ref{alg:krp_sampler_construction} and \ref{alg:krp_sampler_sampling}
in the context of the sampling procedure outlined in 
section \ref{sec:krp_leverage_defs}. We verify the complexity claims first and
then prove correctness.

\textbf{Construction and Update:} For each matrix $U_j$,  
Algorithm \ref{alg:krp_sampler_construction} builds an efficient row sampling data structure $Z_j$ as specified by Lemma 
\ref{lemma:efficient_row_sampler}. We let the p.s.d.\ matrix $Y$ that parameterizes each sampler be a matrix of ones, and we set $F = R$. From Lemma
\ref{lemma:efficient_row_sampler}, the time to construct sampler $Z_j$ is
$O(I_j R^2)$. The space used by sampler $Z_j$ is
$O(\ceil{I_j / F} R^2) = O(I_j R)$, since $F = R$.
In case $I_j < R$, we use the special case described in
Appendix \appendixref{appendix:efficient_lemma_proof} to get a space overhead 
$O(1)$, avoiding a term $O(R^2)$ in the space complexity.

Summing the time and space complexities over all $j$ proves part 1 of the theorem. To update the data structure if
matrix $U_j$ changes, we only need to rebuild sampler $Z_j$ for a cost of
$O(I_j R^2)$. The construction phase also computes and stores the Gram 
matrix $G_j$ for each matrix $U_j$. We defer the update procedure in case a single
entry of matrix $U_j$ changes to Appendix 
\appendixref{appendix:efficient_single_element_updates}.

\textbf{Sampling: } For all indices $k$ (except possibly $j$), lines 1-5 from 
Algorithm \ref{alg:krp_sampler_sampling} compute $G_{>k}$ and its
eigendecomposition. Only a single pass over the Gram
matrices $G_k$ is needed, so these steps cost $O(R^3)$ for each index $k$. Line 5 builds
an efficient row sampler $E_k$ for the matrix of scaled eigenvectors 
$\sqrt{\Lambda_k} \cdot V_k$. For sampler $k$, we set $Y = G_k$ with cutoff
parameter $F = 1$. From Lemma \ref{lemma:efficient_row_sampler}, the construction cost is 
$O(R^3)$ for each index $k$, and the space required by each sampler is $O(R^3)$. Summing these quantities over all $k \neq j$ gives asymptotic runtime $O(NR^3)$ for lines 2-5. 

The loop spanning lines 6-12 draws $J$ row indices from the Khatri-Rao product $U_{\neq j}$. For
each sample, we maintain a ``history vector'' $h$ to write the variables $h_{<k}$ from
Equation \eqref{eq:h_leq_k_defn}. For each index $k \neq j$, we draw random 
variable $\hat u_k$ using the
row sampler $E_k$. This random draw indexes a scaled eigenvector 
of $G_{>k}$. We then use the
history vector $h$ multiplied by the eigenvector to sample a row
index $\hat t_k$ using data structure $Z_k$. The history 
vector $h$ is updated, and we proceed to draw the next index 
$\hat t_k$. 

As written, lines 2-5 also incur
scratch space usage $O(NR^3)$.
The scratch space can
be reduced to $O(R^3)$ by
exchanging the order of loops on
line 6 and line 8 and allocating
$J$ separate history vectors $h$,
once for each draw. Under this
reordering, we perform all $J$ draws for 
each variable $\hat u_k$ and $\hat t_k$
before moving to $\hat u_{k+1}$ and
$\hat t_{k+1}$. In this case, only
a single data structure $E_k$ is required
at each iteration of the outer loop, 
and we can avoid building all
the structures in advance on line 5.
We keep the algorithm in the form
written for simplicity, but we 
implemented the memory-saving approach
in our code.

From Lemma 
\ref{lemma:efficient_row_sampler}, lines 9 and 10 cost $O(R^2 \log R)$ and 
$O\paren{R^2 \log \ceil{I_k / R}}$,
respectively. Line 11 costs $O(R)$ and contributes a lower-order term. Summing over all $k \neq j$,
the runtime to draw a single sample is 
\[
O\paren{\sum_{k \neq j} ( R^2 \log \ceil{I_k / R} + R^2 \log R ) }
= O\paren{\sum_{k \neq j} R^2 \log \max \paren{I_k, R}}.
\]
Adding the runtime for all $J$ samples to the runtime of the loop spanning lines 2-6 gives
runtime $O\paren{NR^3 + J \sum_{k \neq j} R^2 \log \max \paren{I_k, R}}$, and 
the complexity claims have been proven.

\textbf{Correctness: } We show correctness for the case where $j=-1$
and we sample from the Khatri-Rao product of all matrices $U_k$,
since the proof for any other value of $j$ requires a simple 
reindexing of matrices. To show that our sampler is correct, it is enough
to prove the condition that for $1 \leq k \leq N$,
\begin{equation}
p(\hat t_k = t_k\ \vert\ h_{<k}) = q_{h_{<k}, U_k, G_{>k}} \br{t_k},
\label{eq:corr_condition}
\end{equation}
since, by Theorem \ref{thm:malik2022}, 
$p(\hat s_k = s_k\ \vert\ \hat s_{<k} = s_{<k}) = q_{h_{<k}, U_k, G_{>k}} \br{s_k}$.
This would imply that the joint random variable $(\hat t_1, ..., \hat t_{N})$ 
has the same probability distribution as $(\hat s_1, ..., \hat s_{N})$,
which by definition follows the leverage score distribution on
$U_1 \odot ... \odot U_{N}$. To prove 
the condition in Equation \eqref{eq:corr_condition}, we apply Equations \eqref{eq:first_dist}
and \eqref{eq:second_conditional_dist} derived earlier:
\begin{equation}
\begin{aligned}
p(&\hat t_k = t_k\ \vert\ h_{<k}) \\
&=\sum_{u_k=1}^R p(\hat t_k = t_k\ \vert\ \hat u_k = u_k, h_{<k}) p(\hat u_k = u_k
\ \vert\ h_{<k}) &\just{Bayes' Rule} \\
&=\sum_{u_k=1}^R w \br{u_k} \frac{W \br{t_k, u_k}}{\norm{W \br{:, u_k}}_1} 
&\just{Equations \eqref{eq:first_dist} and \eqref{eq:second_conditional_dist}, in
reverse} \\
&= q_{h_{<k}, U_k, G_{>k}} \br{t_k}.
\end{aligned}
\end{equation}
\end{proof}

\subsection{Efficient Single-Element Updates}
\label{appendix:efficient_single_element_updates}
Applications such as CP decomposition typically change all entries of a single matrix $U_j$ between
iterations, incurring an update cost $O(I_j R^2)$
for our data structure from Theorem \ref{thm:main_krp_res}. In case
only a single element of $U_j$ changes, our data structure can be updated
in time $O \paren{R \log I_j}$.
\begin{proof}
Algorithm \ref{alg:singlerow_update} gives the procedure when the update 
$U_j\br{r, c} := \hat u$ is performed. The matrices $G^v$ refer to the partial Gram
matrices maintained by each node $v$ of the segment trees in our data structure, and
the matrix $\tilde U_j$ refers to the matrix $U_j$ before the update operation.
\begin{algorithm}
   \caption{UpdateSampler($j, r, c, \hat u$)} 
   \label{alg:singlerow_update}
\begin{algorithmic}[1]
        \STATE Let $u = \tilde U_j\br{r, c}$
        \STATE Locate $v$ such that $r \in S(v)$
        \STATE Update $G^v\br{c, :} \pluseq (\hat u - u) \tilde U_j \br{r, :}$  
        \STATE Update $G^v\br{:, c} \pluseq (\hat u - u) \tilde U_j \br{r, :}^\top$  
        \STATE Update $G^v\br{c, c} \pluseq (\hat u - u)^2$ 
        \WHILE{$v \neq \textrm{root}(T_{I_j, R})$}
            \STATE $v_{\textrm{prev}} := v$, $v := A(v)$
            \STATE Update $G^v := G^{\textrm{sibling}(v_{\textrm{prev} })} + G^{v_{\textrm{prev}}}$
        \ENDWHILE
\end{algorithmic}
\end{algorithm}

Let $T_{I_j, R}$ be the segment tree 
corresponding to matrix $U_j$ in the data structure, and let $v \in T_{I_j, R}$ be the leaf whose segment contains $r$. Lines 3-5 of the algorithm update the row and column 
indexed by $c$ in the partial Gram matrix held by the leaf node. 

The only other nodes requiring an update are ancestors of $v$, each holding
a partial Gram matrix that is the sum of its two children. Starting from the 
direct parent $A(v)$, the loop on line 6 performs these ancestor updates.
The addition on line 8 only requires time $O(R)$, since only row and column
$c$ change between the old value of $G^v$ and its updated version. 
Thus, the runtime of this procedure is $O(R \log \ceil{I_j  / R})$ 
from multiplying the cost to update a single node by the depth of 
the tree. 
\end{proof}

\subsection{Extension to Sparse Input Matrices}
\label{appendix:sparse_input_extension}
Our data structure is designed to sample from Khatri-Rao products
$U_1 \odot ... \odot U_N$ where the input matrices $U_1, ..., U_N$ are dense,
a typical situation in tensor decomposition. Slight modifications to the
construction procedure permit our data structure to handle sparse matrices
efficiently as well. The following corollary states the result as a 
modification to Theorem \ref{thm:main_krp_res}.

\begin{corollary}[Sparse Input Modification]
\label{cor:sparse_inputs}
When input matrices $U_1, ..., U_N$ are sparse, point 1 of Theorem 
\ref{thm:main_krp_res} can be modified so that the proposed data structure has
$O\paren{R \sum_{j=1}^N \textrm{nnz}(U_j)}$ construction time and 
$O\paren{\sum_{j=1}^N \textrm{nnz}(U_j)}$ storage space. The sampling time
and scratch space usage in point 2
of Theorem \ref{thm:main_krp_res} does 
not change. The single-element
update time in point 1 is likewise unchanged.
\end{corollary}

\begin{proof}
We will modify the data structure in Lemma \ref{lemma:efficient_row_sampler}.
The changes to its construction and storage costs will propagate to
our Khatri-Rao product sampler, which maintains one of these data structures
for each input matrix. 

Let us restrict ourselves to the case $F = R, Y=\br{1}$ in relation to
the data structure in Lemma \ref{lemma:efficient_row_sampler}.
These choices for $F$ and $Y$ are used in the construction phase 
given by Algorithm \ref{alg:krp_sampler_construction}. The proof in
Appendix 
\appendixref{appendix:efficient_lemma_proof} constrains each leaf
$v$ of a segment tree $T_{I, F}$ to hold a contiguous segment 
$S(v) \subseteq \br{I}$ of cardinality at most $F$. Instead, choose each segment 
$S(v) = \set{S_0(v)..S_1(v)}$ so that $U\br{S_0(v):S_1(v), :}$ has
at most $R^2$ nonzeros, and the leaf count of the tree is 
at most $\ceil{\textrm{nnz}(U) / R^2} + 1$ for input matrix 
$U \in \RR^{I \times R}$. Assuming the nonzeros of $U$ are sorted
in row-major order, we can construct such a partition of $\br{I}$ into segments in time $O(\textrm{nnz}(U))$ by iterating in order through the nonzero 
rows and adding each of them to a ``current'' segment. We shift to
a new segment when the current segment cannot hold any more nonzeros.

This completes the modification to the data structure in
Lemma \ref{lemma:efficient_row_sampler}, and we now analyze its
updated time / space complexity.

\textbf{Updated Construction / Update 
Complexity of Lemma 
\ref{lemma:efficient_row_sampler}, $F=R, Y=\br{1}$}: Algorithm 
\ref{alg:row_sampler_construction} constructs the partial Gram matrix
for each leaf node $v$ in the segment tree. Each nonzero in the segment
$U\br{S_0(v):S_1(v), :}$ contributes time $O(R)$ during line 3 of 
Algorithm \ref{alg:row_sampler_construction} to update a single 
row and column of $G^v$. Summed over all leaves, the cost of line 3
is $O(\textrm{nnz}(U) R)$. The remainder of the construction procedure
updates the partial Gram matrices of all internal nodes. Since there
are at most $O\paren{\ceil{\textrm{nnz}(U) / R^2}}$ internal nodes 
and the addition on line 6 costs $O(R^2)$ per node, the remaining 
steps of the construction procedure cost $O(\textrm{nnz}(U))$, a
lower-order term. The construction time is therefore 
$O(\textrm{nnz}(U) R)$.

Since we store a single partial Gram matrix of size $R^2$ at each
of $O\paren{\ceil{\textrm{nnz}(U) / R^2}}$ internal nodes,
the space complexity of our modified data structure is 
$O(\textrm{nnz}(U))$.

Finally, the data structure update time 
in case a single element of $U$ is modified does not change from 
Theorem \ref{thm:main_krp_res}. Since the
depth of the segment tree
$\ceil{\textrm{nnz}(U) / R^2} + 1$ is 
upper-bounded by $\ceil{I / R} + 1$, the
runtime of the update procedure in 
Algorithm \ref{alg:singlerow_update} stays the same.

\textbf{Updated Sampling Complexity of 
Lemma \ref{lemma:efficient_row_sampler}, $F=R, Y=\br{1}$}: 
The procedure
``RowSample'' in Algorithm \ref{alg:row_sampler_sampling} now conducts
a traversal of a tree of depth $O(\ceil{\textrm{nnz}(U) / R^2})$.
As a result, we can still upper-bound 
the number of calls to procedure $\tilde m$
as $\ceil{I / F}$. The runtime of procedure $\tilde m$ is unchanged.
The runtime of procedure $\tilde q$ for leaf node $c$ is dominated 
by the matrix-vector multiplication $U\br{S_0(c):S_1(c), :} \cdot h$.
This runtime is $O\paren{\textrm{nnz} \paren{U\br{S_0(c):S_1(c), :}}} \leq 
O\paren{R^2}$. Putting these facts together, the sampling complexity
of the data structure in Lemma \ref{lemma:efficient_row_sampler} does
not change under our proposed modifications for $F=R, Y=\br{1}$.

\textbf{Updated Construction Complexity of Theorem \ref{thm:main_krp_res}}:
Algorithm \ref{alg:krp_sampler_construction} now requires 
$O\paren{R \sum_{j=1}^N \textrm{nnz}(U_j)}$ construction time and 
$O\paren{\sum_{j=1}^N \textrm{nnz}(U_j)}$ storage space, summing the
costs for the updated structure from Lemma \ref{lemma:efficient_row_sampler}
over all matrices $U_1, ..., U_N$. The sampling complexity of these
data structures is unaffected by the modifications, which completes
the proof of the corollary.
\end{proof}

\subsection{Alternating Least Squares CP Decomposition}
\label{appendix:als_cp_decomp}

\paragraph{CP Decomposition.}
CP decomposition represents an $N$-dimensional tensor
$\tilde{\scr T} \in \RR^{I_1 \times ...
\times I_n}$ 
as a weighted sum of generalized
outer products. Formally, let
$U_1, ..., U_N$ with 
$U_j \in \RR^{I_j \times R}$ be factor
matrices with each column having unit norm, 
and let $\sigma \in \RR^R$ be a nonnegative
coefficient vector. We call $R$ the rank of
the decomposition. The tensor $\tilde{\scr T}$ that
the decomposition represents is given
elementwise by
\[
\tilde{\scr T}\br{i_1, ..., i_N} :=
\gen{\sigma^\top, 
U_1\br{i_1, :}, ..., U_N\br{i_N, :}}
= \sum_{r=1}^R \sigma\br{r} U_1[i_1, r] \cdots U_N[i_N, r],
\]
which is a generalized inner product between
$\sigma^\top$ and rows $U_j\br{i_j, :}$ for 
$1 \leq j \leq N$.
Given an input tensor $\scr T$ and a target
rank $R$, the goal of approximate CP decomposition
is to find a rank-$R$ representation $\tilde{\scr T}$
that minimizes the Frobenius norm
$\norm{\scr T - \tilde{\scr T}}_F$.

\paragraph{Definition of Matricization.}
The matricization $\textrm{mat}(\scr T, j)$ flattens tensor 
$\scr T \in \RR^{I_1 \times ... \times I_N}$ into 
a matrix and isolates mode $j$ along the row axis of the output. The output
of matricization has dimensions $I_j \times \prod_{k \neq j} I_k$.
We take the formal definition below from a survey by \citet{kolda_tensor_overview}.
The tensor entry $\scr T\br{i_1, ..., i_N}$ is equal to the matricization
entry $\textrm{mat}(\scr T, j)\br{i_j, u}$, where
\begin{equation*}
u = 1 + \sum_{\substack{k=1 \\ k \neq j}}^N (i_k - 1) \prod_{\substack{m=1 \\ m\neq j}}^{k-1} I_m.
\end{equation*}

\paragraph{Details about Alternating
Least Squares.}
Let $U_1, ..., U_N$ be factor matrices of a low-rank CP
decomposition, $U_k \in \RR^{I_k \times R}$. We use
$U_{\neq j}$ to denote $\bigodot_{k=N, k \neq j}^{k=1} U_k$. Note the
inversion of order here to match
indexing in the definition of
matricization above. 
Algorithm \ref{alg:als_algorithm} gives
the non-randomized alternating least squares algorithm
CP-ALS that produces a decomposition of target rank $R$ given input 
tensor $\scr T \in \RR^{I_1 \times ... \times I_N}$
in general format. The random initialization on line 1 of
the algorithm can be
implemented by drawing each entry of the factor matrices $U_j$ 
according to a standard normal distribution, or via a randomized range 
finder \cite{halko_rrf}. The vector $\sigma$ stores the generalized
singular values of the decomposition. At iteration $j$ within a round, ALS
holds all factor matrices except $U_j$ constant and solves a linear-least
squares problem on line 6 for a new value for $U_j$. In between least squares
solves, the algorithm renormalizes the columns of each matrix $U_j$ to unit
norm and stores their original norms in the vector $\sigma$. Appendix 
\appendixref{appendix:sparse_cp}
contains more 
details about the randomized range finder and the convergence criteria 
used to halt iteration.

\begin{algorithm}
   \caption{CP-ALS($\scr T$, $R$)} 
   \label{alg:als_algorithm}
\begin{algorithmic}[1]
        \STATE Initialize $U_j \in \RR^{I_j \times R}$
        randomly for $1 \leq j \leq N$.

        \STATE Renormalize $U_j\br{:, i} \diveq \norm{U_j\br{:, i}}_2 $, 
        $1 \leq j \leq N, 1 \leq i \leq R$.

        \STATE Initialize $\sigma \in \RR^{R}$ to $\br{1}$. 

        \WHILE{not converged}
            \FOR{$j=1...N$}
                \STATE $U_j := \argmin_X \norm{U_{\neq j} \cdot X^\top 
                -  \textrm{mat}(\scr T, j)^\top}_F $

                \STATE $\sigma \br{i} = \norm{U_j \br{:, i}}_2$, $1 \leq i \leq R$

                \STATE Renormalize $U_j\br{:, i} \diveq \norm{U_j\br{:, i}}_2 $, 
                $1 \leq i \leq R$.
            \ENDFOR
        \ENDWHILE

        \STATE \textbf{return} $\br{\sigma; U_1, ..., U_N}$.
\end{algorithmic}
\end{algorithm}
We obtain a randomized algorithm for sparse tensor CP decomposition by replacing
the exact least squares solve on line 6 with a randomized method according to
Theorem \ref{thm:lev_score_lowerbounds}. Below, we prove
Corollary \ref{cor:downsampled_cpals_complexity}, which derives the complexity
of the randomized CP decomposition algorithm.

\begin{proof}[Proof of Corollary \ref{cor:downsampled_cpals_complexity}]
The design matrix $U_{\neq j}$ for optimization problem $j$ within a round of ALS has dimensions
$\prod_{k \neq j} I_k \times R$. The observation matrix 
$\textrm{mat}(\scr T, j)^\top$ has dimensions 
$\prod_{k \neq j} I_k \times I_j$.  To achieve error threshold 
$1 + \varepsilon$ with probability $1 - \delta$ on each solve, 
we draw $J = \tilde O\paren{R / (\varepsilon \delta)}$ rows from both the 
design and observation matrices and solve the downsampled problem
(Theorem \ref{thm:lev_score_lowerbounds}). These rows
are sampled according to the leverage score distribution on the rows of $U_{\neq j}$, for
which we use the data structure in Theorem \ref{thm:main_krp_res}. After a 
one-time initialization cost $O(\sum_{j=1}^{N} I_j R^2)$) before the ALS 
iteration begins, the complexity to draw $J$ samples (assuming $I_j \geq R$) is 
\[
O\paren{N R^3 + J\sum_{k \neq j} R^2 \log I_k}
=
\tilde O\paren{N R^3 + \frac{R}{\varepsilon \delta} \sum_{k \neq j} R^2 \log I_k}.
\]
The cost to assemble the corresponding
subset of the observation matrix is $O(J I_j) = \tilde O(R I_j / (\varepsilon \delta))$.
The cost to solve the downsampled least squares problem is $O(J R^2) = 
\tilde O(I_j R^2 / (\varepsilon \delta))$, which dominates the cost 
of forming the subset of the observation matrix. Finally, we require additional time
$O(I_j R^2)$ to update the sampling data structure 
(Theorem \ref{thm:main_krp_res} part 1). Adding these terms together and summing 
over $1 \leq j \leq N$ gives
\begin{equation}
\begin{aligned}
\tilde O &\paren{\frac{1}{\varepsilon \delta} \cdot \sum_{j=1}^N \br{
I_j R^2 + \sum_{k \neq j} R^3 \log I_k}} \\
=\tilde O&\paren{\frac{1}{\varepsilon \delta} \cdot \sum_{j=1}^N \br{
I_j R^2 + (N-1) R^3 \log I_j}}. \\
\end{aligned}
\end{equation}
Rounding $N-1$ to $N$ and multiplying by the number of iterations gives the desired complexity.
When $I_j < R$ for any $j$, the complexity changes in Theorem \ref{thm:main_krp_res}
propagate to the equation above. The column renormalization on line 8 of the 
CP-ALS algorithm contributes additional time
$O\paren{\sum_{j=1}^N I_j R}$ per round, a lower-order term. 

\end{proof}
\subsection{Experimental Platform and Sampler Parallelism}
\label{appendix:efficient_parallel_impl}
We provide two implementations of our sampler. The first
is a slow reference implementation written entirely in Python, which closely mimics our
pseudocode and can be used to test correctness. The second is an efficient implementation written in C++, parallelized in shared memory with OpenMP and Intel Thread Building Blocks.

Each Perlmutter CPU node (our experimental platform) is equipped with two sockets, each 
containing an AMD EPYC 7763 processor with 64 cores. All benchmarks were conducted with our 
efficient C++ implementation using 128 OpenMP threads. We link our code against 
Intel Thread Building blocks to call a multithreaded sort function when decomposing sparse 
tensors. We use OpenBLAS 0.3.21 to handle linear algebra with OpenMP parallelism enabled, 
but our code links against any 
linear algebra library implementing the CBLAS and LAPACKE interfaces.

Our proposed data structure samples from the exact distribution of
leverage scores of the Khatri-Rao product, thereby enjoying better sample efficiency 
than alternative approaches such as CP-ARLS-LEV \cite{larsen_practical_2022}. 
The cost to draw each sample, however, is $O(R^2 \log H)$, where $H$ is the number of
rows in the Khatri-Rao product. Methods such as row-norm-squared sampling or CP-ARLS-LEV
can draw each sample in time $O(\log H)$ after appropriate preprocessing. Therefore,
efficient parallelization of our sampling procedure is required for competitive performance,
and we present two strategies below.

\begin{enumerate}
    \item \textbf{Asynchronous Thread Parallelism}: The KRPSampleDraw procedure in Algorithm 
    \ref{alg:krp_sampler_sampling} can be called by multiple threads concurrently without
    data races. The simplest parallelization strategy divides the $J$ samples 
    equally among the threads in a team, each of which makes calls to KRPSampleDraw
    asynchronously. This strategy works well on a CPU, but is less attractive on a SIMT processor like a GPU where 
    instruction streams cannot diverge without significant
    performance penalties.

    \item \textbf{Synchronous Batch Parallelism} As an alternative to the asynchronous
    strategy, suppose for the moment that all leaves have the same depth in each segment tree. Then for every sample, 
    STSample makes a sequence of calls to $\tilde m$, each updating the current
    node by branching left or right in the tree. The length of this sequence is the depth of the tree, and it is followed by a single call to the function $\tilde q$. 
    Observe that procedure $\tilde m$ in Algorithm
    \ref{alg:row_sampler_sampling} can be computed
    with a matrix-vector multiplication
    followed by a dot product. The procedure $\tilde q$ of
    Algorithm \ref{alg:row_sampler_sampling} 
    requires the same two operations if
    $F = 1$ or $Y= \br{1}$. Thus, we can create
    a batched version of our sampling procedure that makes
    a fixed length sequence of calls to batched \verb|gemv|
    and \verb|dot| routines. All processors march in lock-step 
    down the levels of each segment tree, each tracking the branching paths of a
    distinct set of samples. The MAGMA 
    linear algebra library provides a
    batched version of \verb|gemv| \cite{magma_batched}, 
    while a batched dot
    product can be implemented with an ad hoc kernel. MAGMA also offers a batched version
    of the symmetric rank-$k$ update routine \verb|syrk|, 
    which is helpful to parallelize row sampler construction 
    (Algorithm \ref{alg:row_sampler_construction}). When
    all leaves in the tree are not at the same level, the
    the bottom level of the tree can be handled with
    a special sequence of instructions making the 
    required additional calls to $\tilde m$. 
\end{enumerate}
Our CPU code follows the batch synchronous design pattern. 
To avoid dependency on GPU-based MAGMA routines in our CPU prototype, 
portions of the code that should be batched BLAS calls are
standard BLAS calls wrapped in a \verb|for| loop. These sections can be easily replaced when the 
appropriate batched routines are available.

\subsection{Sparse Tensor CP Experimental Configuration}
\label{appendix:sparse_cp}
\begin{table}[h]
\caption{Sparse Tensors from FROSTT collection.}
\label{tab:frostt_datasets}
\vskip 0.15in
\begin{center}
\begin{small}
\begin{sc}
\begin{tabular}{lccll}
\toprule
Tensor & Dimensions & Nonzeros & Prep. & 
Init. \\
\midrule
Uber Pickups & 183 $\times$ 24 $\times$ 1,140 $\times$ 1,717 & 3,309,490 & None & IID \\
Enron Emails & 
6,066 $\times$ 5,699 $\times$ 244,268 $\times$ 1,176
& 54,202,099 & log & RRF \\
NELL-2 & 12,092 $\times$ 9,184 $\times$ 28,818 &
76,879,419 & log & IID \\
Amazon Reviews & 4,821,207 $\times$ 1,774,269 $\times$ 1,805,187 & 1,741,809,018 & None & IID \\ 
Reddit-2015 & 8,211,298 $\times$ 176,962 $\times$ 8,116,559 & 4,687,474,081 & log & IID \\ 
\bottomrule
\end{tabular}
\end{sc}
\end{small}
\end{center}
\vskip -0.1in
\end{table}
Table \ref{tab:frostt_datasets} lists the nonzero counts and dimensions of sparse tensors
in our experiments \cite{frosttdataset}. We took the
log of all values in the Enron, NELL-2, and Reddit-2015
tensors. Consistent with established practice, 
this operation damps the effect of a few 
high magnitude tensor entries 
on the fit metric \cite{larsen_practical_2022}. 

The factor matrices for the Uber, Amazon, NELL-2, and Reddit experiments were initialized with i.i.d.\ entries from the 
standard normal distribution. As suggested by 
\citet{larsen_practical_2022}, the Enron tensor's 
factors were initialized 
with a randomized range finder
\cite{halko_rrf}. The 
range finder algorithm initializes 
each factor 
matrix $U_j$ as 
$\textrm{mat}(\scr T, j) S$, a sketch
applied to the mode-$j$ matricization
of $\scr T$ with 
$S \in \RR^{\prod_{k \neq j} I_k \times R}$. Larsen and Kolda
chose $S$ as a sparse sampling matrix to
select a random subset of fibers
along each mode. We instead used 
an i.i.d.
Gaussian sketching matrix that was
not materialized explicitly. Instead,
we exploited the sparsity of $\scr T$
and noted that at most
$\textrm{nnz}\paren{\scr T}$ columns
of $\textrm{mat}(\scr T, j)$ were 
nonzero. Thus, we computed 
at most $\textrm{nnz}\paren{\scr T}$
rows of the random sketching matrix 
$S$, which were
lazily generated and 
discarded during the matrix
multiplication without
incurring excessive memory overhead. 

ALS was run for a maximum of 40 rounds 
on all tensors except for Reddit, which was run for 80 rounds. The
exact fit was computed every 5 rounds (defined as 1 epoch), and we used an early stopping condition to terminate
runs before the maximum round count. The algorithm was terminated at epoch $T$ if the maximum fit 
in the last 3 epochs did 
not exceed the maximum fit from epoch 1 through epoch $T-3$ by tolerance $10^{-4}$. 

Hybrid CP-ARLS-LEV deterministically includes rows from the Khatri-Rao product whose
probabilities exceed a threshold $\tau$. The ostensible
goal of this procedure is to improve diversity in sample
selection, as CP-ARLS-LEV may suffer from many repeat draws 
of high probability rows. 
We replicated the conditions proposed in
the original work by selecting $\tau = 1 / J$ \cite{larsen_practical_2022}.

Individual trials of non-randomized (exact) ALS 
on the Amazon and
Reddit tensors required several hours on a
single Perlmutter CPU node. To speed up our 
experiments, accuracy measurements for exact ALS 
in Figure \ref{fig:accuracy_bench} were carried 
out using multi-node SPLATT, The 
Surprisingly ParalleL spArse Tensor Toolkit 
\cite{splattsoftware}, on four Perlmutter 
CPU nodes. The fits computed by SPLATT agree with 
those computed by our own non-randomized ALS 
implementation. As a result, Figure 
\ref{fig:accuracy_bench} verifies that our randomized
algorithm STS-CP produces tensor decompositions with
accuracy comparable to those by 
highly-optimized, 
state-of-the-art CP decomposition software.
We leave a distributed-memory
implementation of our \textit{randomized}
algorithms to future work.

\subsection{Efficient Computation of Sketch Distortion}
\label{appendix:distortion}
The definition of $\sigma$ in this section is different from its
definition in the rest of this work. The condition number 
$\kappa$ of a matrix $M$ is defined as
\[
\kappa(M) := \frac{\sigma_\textrm{max}(M)}{\sigma_\textrm{min}(M)}
\]
where $\sigma_\textrm{min}(M)$ and $\sigma_\textrm{max}(M)$ denote
the minimum and maximum nonzero singular values of $M$. 
Let $A$ be a Khatri-Rao product of $N$ matrices $U_1, ..., U_N$ with $\prod_{j=1}^N I_j$ rows, $R$ columns, and rank $r \leq R$. Let 
$A = Q \Sigma V^\top$ be its reduced singular value decomposition
with $Q \in \RR^{\prod_j I_j \times r}, \Sigma \in \RR^{r \times r}$,
and $V \in \RR^{r \times R}$. Finally, let 
$S \in \RR^{J \times \prod_{j} I_j}$ be a leverage score
sampling matrix for $A$. Our goal is to 
compute $\kappa(SQ)$ without fully materializing either 
$A$ or its QR decomposition. We derive
\begin{equation}
    \begin{aligned}
\kappa(SQ)
&= \kappa(SQ \Sigma V^\top V \Sigma^{-1}) \\
&= \kappa(S A V \Sigma^{-1}) \\
    \end{aligned}
    \label{eq:kappa_derivation}
\end{equation}
The matrix $SA \in \RR^{J \times R}$ is efficiently computable 
using our leverage score
sampling data structure. We require time $O(JR^2)$ to multiply it
by $V \Sigma^{-1}$ and compute the singular value decomposition of
the product to get the condition number. Next observe 
that $A^\top A = V \Sigma^2 V^\top$, so we can recover $V$ and 
$\Sigma^{-1}$ by eigendecomposition of 
$A^\top A \in \RR^{R \times R}$ in time $O(R^3)$. 
Finally, recall the formula 
\[
A^\top A = \startimes_{j=1}^N U_j^\top U_j 
\]
used at the beginning of Section \ref{sec:efficient_krp_sampling} that
enables computation of $A^\top A$ in time $O \paren{\sum_{j=1}^N I_j R^2}$
without materializing the full Khatri-Rao product. Excluding the time
to form $SA$ (which is given by Theorem \ref{thm:main_krp_res}),
$\kappa(SQ)$ is computable in time 
\[
O \paren{
JR^2 + R^3 + \sum_{j=1}^N I_j R^2}
.
\]
Plugging $\kappa(SQ)$ into Equation \eqref{eq:distortion_defn} gives an
efficient method to compute the distortion.
\subsection{Supplementary Results}
\label{appendix:supp_results}

\subsubsection{Comparison Against Standard
CP Decomposition Packages}
\label{appendix:cp_baseline_comparison}
Table \ref{tab:baseline-comparison} compares
the runtime per ALS round for our algorithm 
against existing common software packages
for sparse tensor CP decomposition. We
compared our algorithm against 
Tensorly version 0.81 \cite{tensorly} 
 and Matlab Tensor Toolbox 
version 3.5 \cite{ttoolbox_sparse}. We 
compared our algorithm against both
non-randomized ALS and a version of CP-ARLS-LEV
in Tensor Toolbox. 

As demonstrated by 
Table \ref{tab:baseline-comparison}, our
implementation exhibits more than 1000x speedup over 
Tensorly and 295x over
Tensor Toolbox (non-randomized) for the
NELL-2 tensor. STS-CP enjoys a dramatic
speedup over Tensorly because the latter
explicitly materializes the Khatri-Rao
product, which is prohibitively 
expensive given the large
tensor dimensions 
(see Table \ref{tab:frostt_datasets}).

STS-CP consistently exhibits
at least 2.5x speedup over 
the version of CP-ARLS-LEV in Tensor Toolbox,
with more than 10x speedup on the Amazon tensor.
To ensure a fair comparison with CP-ARLS-LEV,
we wrote an improved implementation
in C++ that was used for all other experiments.

\begin{table}
\caption{Average time (seconds) per ALS round 
for our method vs. standard 
CP decomposition packages. OOM indicates an 
out-of-memory error. All experiments were 
conducted on a single LBNL Perlmutter CPU
node. Randomized algorithms were benchmarked 
with $2^{16}$ samples per least-squares solve.}
\label{tab:baseline-comparison}
\begin{center}
\begin{small}
\begin{sc}
\begin{tabular}{llllll}  
    \toprule
    Method & Uber & Enron & NELL-2 & Amazon & Reddit \\ 
    \midrule
    Tensorly, Sparse Backend & 64.2 & OOM & 759.6 & OOM & OOM \\
    Matlab TToolbox Standard & 11.6 & 294.4 & 177.4 & >3600 & OOM \\
    Matlab TToolbox CP-ARLS-LEV & 0.5 & 1.4 & 1.9 & 34.2 & OOM \\
    \bf{STS-CP (ours)} & \bf{0.2} & \bf{0.5} &
    \bf{0.6} & \bf{3.4} & \bf{26.0} \\
    \bottomrule
\end{tabular} 
\end{sc}
\end{small}
\end{center}
\end{table}

\subsubsection{Probability Distribution Comparison}
Figure \ref{fig:distribution_comparison} provides confirmation on a small test problem
that our sampler works as expected. For the Khatri-Rao product of three matrices
$A = U_1 \odot U_2 \odot U_3$, it plots the true distribution of leverage scores 
against a normalized histogram of 50,000 draws from the data structure in Theorem
\ref{thm:main_krp_res}. We choose $U_1, U_2, U_3 \in \RR^{8 \times 8}$ initialized i.i.d.\ from a standard normal distribution with
1\% of all entries multiplied by 10. We observe excellent
agreement between the histogram and the true distribution. 

\begin{figure}[h]
    \centering
    \includegraphics[scale=0.35]{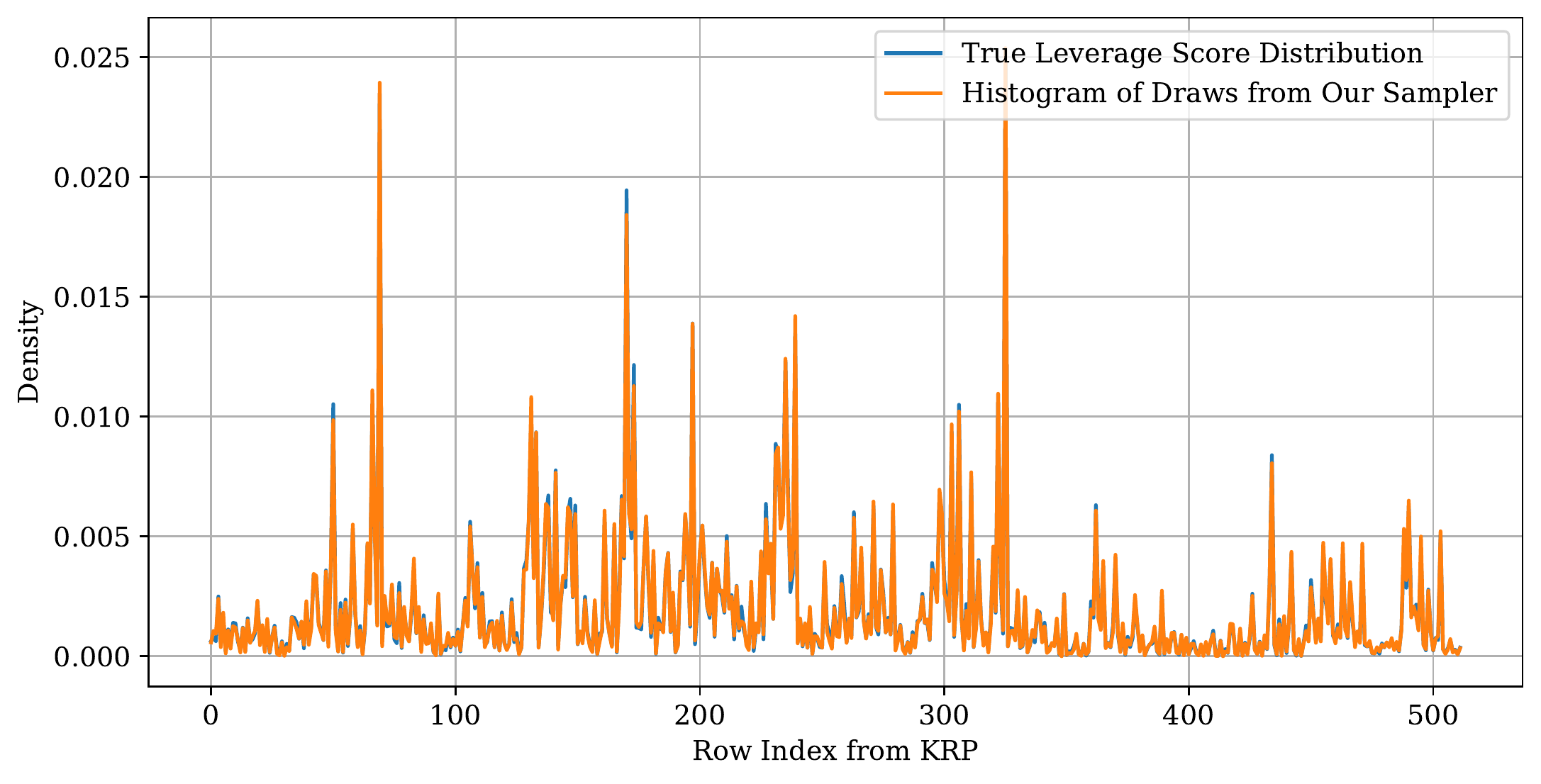}
    \caption{Comparison of true leverage score distribution with histogram of
    50,000 samples drawn from $U_1 \odot U_2 \odot U_3$.}
    \label{fig:distribution_comparison}
\end{figure}

\subsubsection{Fits Achieved for $J=2^{16}$}
Table \ref{tab:accuracies} gives the fits achieved
for sparse tensor decomposition for varying rank and
algorithm (presented graphically in Figure 
\ref{fig:accuracy_comparison}). Uncertainties are 
one standard deviation across 8 runs of ALS.

\begin{table}[h!]
\caption{Fits Achieved by Randomized Algorithms for 
Sparse Tensor Decomposition,
$J = 2^{16}$, and non-randomized ALS. 
The best result among randomized algorithms is 
bolded. ``CP-ARLS-LEV-H'' refers to the hybrid version
of CP-ARLS-LEV and ``Exact'' refers to non-randomized
ALS.}
\label{tab:accuracies}
\vskip 0.15in
\begin{center}
\begin{small}
\begin{sc}
\begin{tabular}{lllll|l}
\\
\toprule
Tensor	& $R$	& CP-ARLS-LEV	& CP-ARLS-LEV-H	& STS-CP (ours)	& Exact	\\
\midrule
\multirow{5}{*}{Uber}	& 25	& .187 $\pm$ 2.30e-03	& .188 $\pm$ 2.11e-03	& \textbf{.189} $\pm$ 1.52e-03	& .190 $\pm$ 1.41e-03	\\
	& 50	& .211 $\pm$ 1.72e-03	& .212 $\pm$ 1.27e-03	& \textbf{.216} $\pm$ 1.18e-03	& .218 $\pm$ 1.61e-03	\\
	& 75	& .218 $\pm$ 1.76e-03	& .218 $\pm$ 2.05e-03	& \textbf{.230} $\pm$ 9.24e-04	& .232 $\pm$ 9.29e-04	\\
	& 100	& .217 $\pm$ 3.15e-03	& .217 $\pm$ 1.69e-03	& \textbf{.237} $\pm$ 2.12e-03	& .241 $\pm$ 1.00e-03	\\
	& 125	& .213 $\pm$ 1.96e-03	& .213 $\pm$ 2.47e-03	& \textbf{.243} $\pm$ 1.78e-03	& .247 $\pm$ 1.52e-03	\\
\midrule
\multirow{5}{*}{Enron}	& 25	& .0881 $\pm$ 1.02e-02	& .0882 $\pm$ 9.01e-03	& \textbf{.0955} $\pm$ 1.19e-02	& .0978 $\pm$ 8.50e-03	\\
	& 50	& .0883 $\pm$ 1.72e-02	& .0920 $\pm$ 6.32e-03	& \textbf{.125} $\pm$ 1.03e-02	& .132 $\pm$ 1.51e-02	\\
	& 75	& .0899 $\pm$ 6.10e-03	& .0885 $\pm$ 6.39e-03	& \textbf{.149} $\pm$ 1.25e-02	& .157 $\pm$ 4.87e-03	\\
	& 100	& .0809 $\pm$ 1.26e-02	& .0787 $\pm$ 1.00e-02	& \textbf{.164} $\pm$ 5.90e-03	& .176 $\pm$ 4.12e-03	\\
	& 125	& .0625 $\pm$ 1.52e-02	& .0652 $\pm$ 1.00e-02	& \textbf{.182} $\pm$ 1.04e-02	& .190 $\pm$ 4.35e-03	\\
\midrule
\multirow{5}{*}{NELL-2}	& 25	& .0465 $\pm$ 9.52e-04	& .0467 $\pm$ 4.61e-04	& \textbf{.0470} $\pm$ 4.69e-04	& .0478 $\pm$ 7.20e-04	\\
	& 50	& .0590 $\pm$ 5.33e-04	& .0593 $\pm$ 4.34e-04	& \textbf{.0608} $\pm$ 5.44e-04	& .0618 $\pm$ 4.21e-04	\\
	& 75	& .0658 $\pm$ 6.84e-04	& .0660 $\pm$ 3.95e-04	& \textbf{.0694} $\pm$ 2.96e-04	& .0708 $\pm$ 3.11e-04	\\
	& 100	& .0700 $\pm$ 4.91e-04	& .0704 $\pm$ 4.48e-04	& \textbf{.0760} $\pm$ 6.52e-04	& .0779 $\pm$ 5.09e-04	\\
	& 125	& .0729 $\pm$ 8.56e-04	& .0733 $\pm$ 7.22e-04	& \textbf{.0814} $\pm$ 5.03e-04	& .0839 $\pm$ 8.47e-04	\\
\midrule
\multirow{5}{*}{Amazon}	& 25	& .338 $\pm$ 6.63e-04	& .339 $\pm$ 6.99e-04	& \textbf{.340} $\pm$ 6.61e-04	& .340 $\pm$ 5.78e-04	\\
	& 50	& .359 $\pm$ 1.09e-03	& .360 $\pm$ 8.04e-04	& \textbf{.366} $\pm$ 7.22e-04	& .366 $\pm$ 1.01e-03	\\
	& 75	& .367 $\pm$ 1.82e-03	& .370 $\pm$ 1.74e-03	& \textbf{.382} $\pm$ 9.13e-04	& .382 $\pm$ 5.90e-04	\\
	& 100	& .366 $\pm$ 3.05e-03	& .371 $\pm$ 2.53e-03	& \textbf{.392} $\pm$ 6.67e-04	& .393 $\pm$ 5.62e-04	\\
	& 125	& .358 $\pm$ 6.51e-03	& .364 $\pm$ 4.22e-03	& \textbf{.400} $\pm$ 3.67e-04	& .401 $\pm$ 3.58e-04	\\
\midrule
\multirow{5}{*}{Reddit}	& 25	& .0581 $\pm$ 1.02e-03	& .0583 $\pm$ 2.78e-04	& \textbf{.0592} $\pm$ 3.07e-04	& .0596 $\pm$ 4.27e-04	\\
	& 50	& .0746 $\pm$ 1.03e-03	& .0738 $\pm$ 4.85e-03	& \textbf{.0774} $\pm$ 7.88e-04	& .0783 $\pm$ 2.60e-04	\\
	& 75	& .0845 $\pm$ 1.64e-03	& .0849 $\pm$ 8.96e-04	& \textbf{.0909} $\pm$ 5.49e-04	& .0922 $\pm$ 3.69e-04	\\
	& 100	& .0904 $\pm$ 1.35e-03	& .0911 $\pm$ 1.59e-03	& \textbf{.101} $\pm$ 6.25e-04	& .103 $\pm$ 7.14e-04	\\
	& 125	& .0946 $\pm$ 2.13e-03	& .0945 $\pm$ 3.17e-03	& \textbf{.109} $\pm$ 7.71e-04	& .111 $\pm$ 7.98e-04	\\
\bottomrule
\end{tabular}
\end{sc}
\end{small}
\end{center}
\vskip -0.1in
\end{table}

\begin{figure}[ht]
     \centering
     \begin{subfigure}[b]{0.48\textwidth}
         \centering
         \includegraphics[scale=0.37]{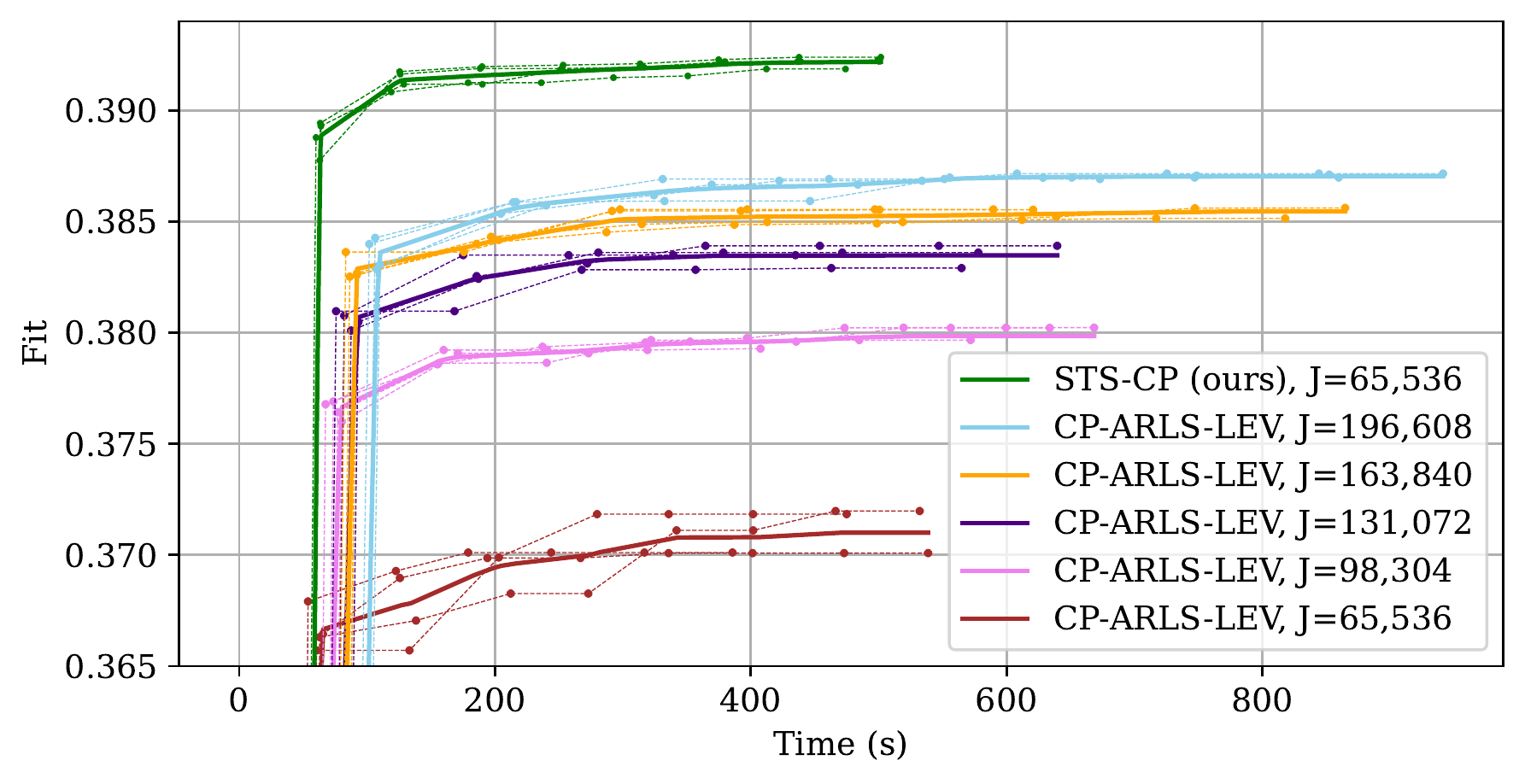}
         \caption{Amazon}
         \label{fig:amazon_fit_function_time}
     \end{subfigure}
     \hfill
     \begin{subfigure}[b]{0.48\textwidth}
         \centering
         \includegraphics[scale=0.37]{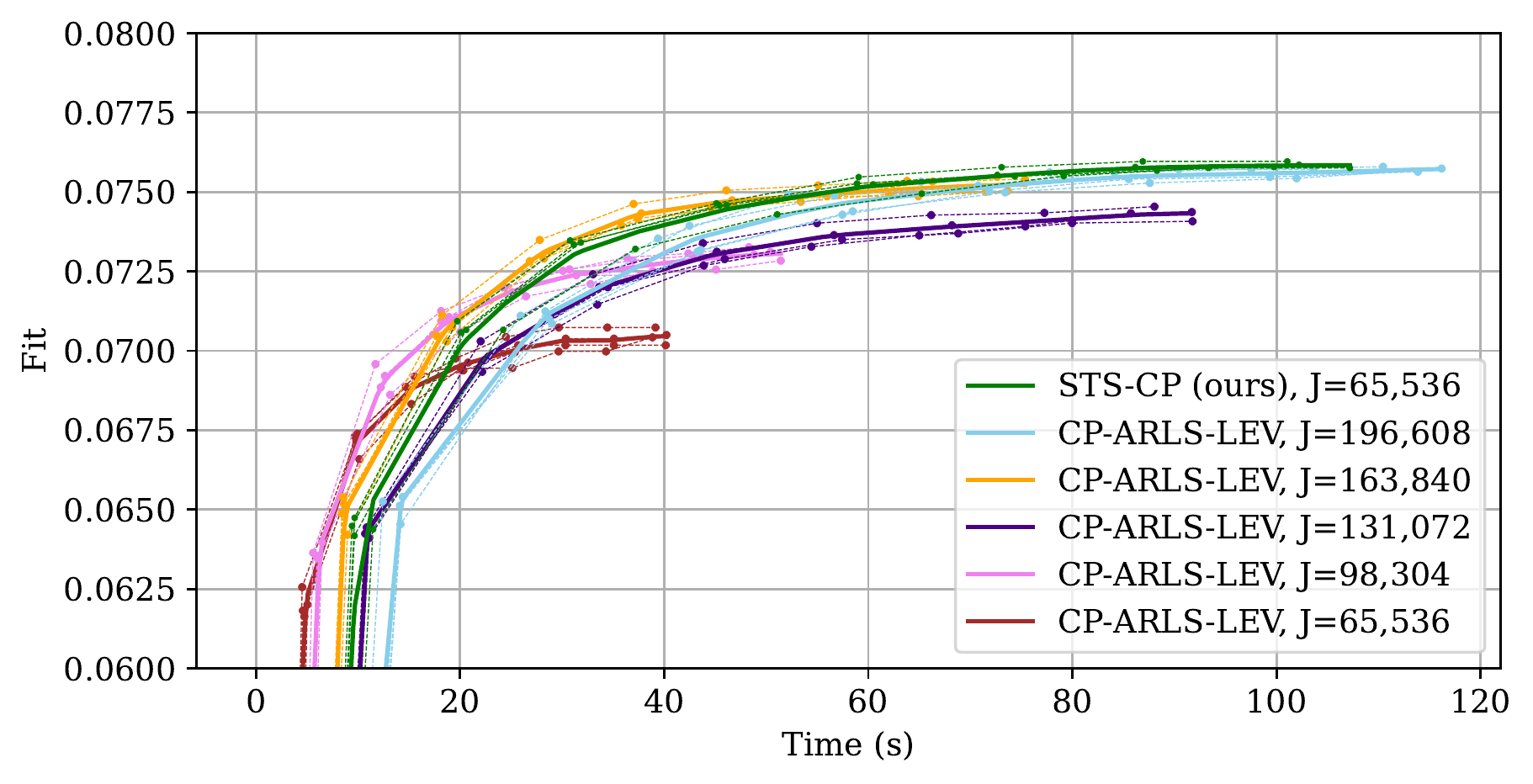}
         \caption{NELL-2}
        \label{fig:nell2_fit_function_time}
     \end{subfigure}
        \caption{Fit as a function of time, $R=100$.}
\end{figure}

\subsubsection{Fit as a Function of Time}
\label{appendix:extra_fit_times}
Figures \ref{fig:amazon_fit_function_time} and \ref{fig:nell2_fit_function_time}
shows the fit as a function of time for the Amazon Reviews and NELL2 tensors.
The hybrid version of CP-ARLS-LEV was used for comparison in both experiments. 
As in section \ref{sec:sparse_tensor_cp_experiments}, thick lines are averages
of the running max fit across 4 ALS trials, shown by the thin dotted lines. For Amazon, 
the STS-CP algorithm makes faster progress than CP-ARLS-LEV at all tested sample counts.

For the NELL-2 tensor, STS-CP makes slower progress 
than CP-ARLS-LEV for sample counts up to $J = 163,840$. On average, these trials
with CP-ARLS-LEV do not achieve the same final fit as STS-CP. 
CP-ARLS-LEV finally achieves a comparable fit to STS-CP when the former uses
$J = 196,608$ samples, compared to $J=65,536$ for our method. 

\subsubsection{Speedup of STS-CP and Practical Usage Guide}
\label{appendix:sts-cp-speedup}

\paragraph{Timing Comparisons.}
For each tensor, we now compare hybrid CP-ARLS-LEV and STS-CP on 
the time required to achieve a fixed fraction of the fit 
achieved by non-randomized ALS. For each tensor and rank in the set
$\set{25, 50, 75, 100, 125}$, we ran both algorithms using a range
of sample counts. We tested STS-CP on values of $J$ from the set
$\set{2^{15}x\ \vert\ 1 \leq x \leq 4}$ for all tensors. 
CP-ARLS-LEV required a sample count that varied significantly
between datasets to hit the required thresholds, and we report 
the sample counts that we tested in Table 
\ref{tab:cp_arls_lev_sample_counts}. Because
CP-ARLS-LEV has poorer sample complexity than STS-CP, 
we tested a wider range of sample counts for
the former algorithm.

\begin{table}[h]
\caption{Tested Sample Counts for hybrid CP-ARLS-LEV}
\label{tab:cp_arls_lev_sample_counts}
\vskip 0.15in
\begin{center}
\begin{small}
\begin{sc}
\begin{tabular}{lcc}
\toprule
Tensor & Values of $J$ Tested \\
\midrule
Uber & $\set{2^{15} x\ \vert\ x \in \set{1..13}}$ \\
Enron & $\set{2^{15} x\ \vert\ x \in \set{1..7} \cup 
\set{10, 12, 14, 16, 18, 20, 22, 26, 30, 34, 38, 42, 46, 50, 54}}$ \\
NELL-2 & $\set{2^{15} x\ \vert\ x \in \set{1..7}}$ \\
Amazon & $\set{2^{15} x\ \vert\ x \in \set{1..7}}$ \\
Reddit & $\set{2^{15} x\ \vert\ x \in \set{1..12}}$ \\
\bottomrule
\end{tabular}
\end{sc}
\end{small}
\end{center}
\vskip -0.1in
\end{table}

For each configuration of tensor, target rank $R$, 
sampling algorithm, and sample count $J$, 
we ran 4 trials using the configuration and 
stopping criteria in Appendix \appendixref{appendix:sparse_cp}. The result
of each trial was a set of $(\textrm{time}, \textrm{fit})$
pairs. For each configuration, we linearly interpolated 
the pairs for each trial and averaged the 
resulting continuous functions over all trials. The
result for each configuration was a function 
$f_{\scr T, R, A, J}: \RR^+ \rightarrow \br{0, 1}$. The
value $f_{\scr T, R, A, J} (t)$ is the average fit at time
$t$ achieved by algorithm $A$ to decompose tensor $\scr T$
with target rank $R$ using $J$ samples per least squares
solve. Finally, let  
\[
\textrm{Speedup}_{\scr T, R, M} := 
\frac{
\min_{J} \textrm{argmin}_{t \geq 0} 
\br{f_{\scr T, R, \textrm{CP-ARLS-LEV-H}, J}(t) > P}
}{
\min_{J} \textrm{argmin}_{t \geq 0} 
\br{f_{\scr T, R, \textrm{STS-CP}, J}(t) > P}
}
\]
be the speedup of STS-CP to over CP-ARLS-LEV (hybrid)
to achieve a threshold fit $P$ on tensor $\scr T$ with
target rank $R$. We let the threshold $P$ for each tensor
$\scr T$ be a fixed fraction of the fit achieved by
non-randomized ALS (see Table \ref{tab:accuracies}).

Figure \ref{fig:tensor_speedup_no_enron} reports the
speedup of STS-CP over hybrid CP-ARLS-LEV for $P=0.95$
on all tensors except Enron. For large tensors with
over one billion nonzeros, we report a significant
speedup anywhere from 1.4x to 2.0x for all tested ranks.
For smaller tensors with less than 100 million nonzeros, 
the lower cost of each least squares solve lessens the
impact of the expensive, more accurate sample selection
phase of STS-CP. Despite this, STS-CP performs
comparably to CP-ARLS-LEV at most ranks, with 
significant slowdown only at rank 25 on the 
smallest tensor Uber. 

\begin{figure}[ht]
     \centering
     \begin{subfigure}[b]{0.48\textwidth}
         \centering
         \includegraphics[scale=0.37]{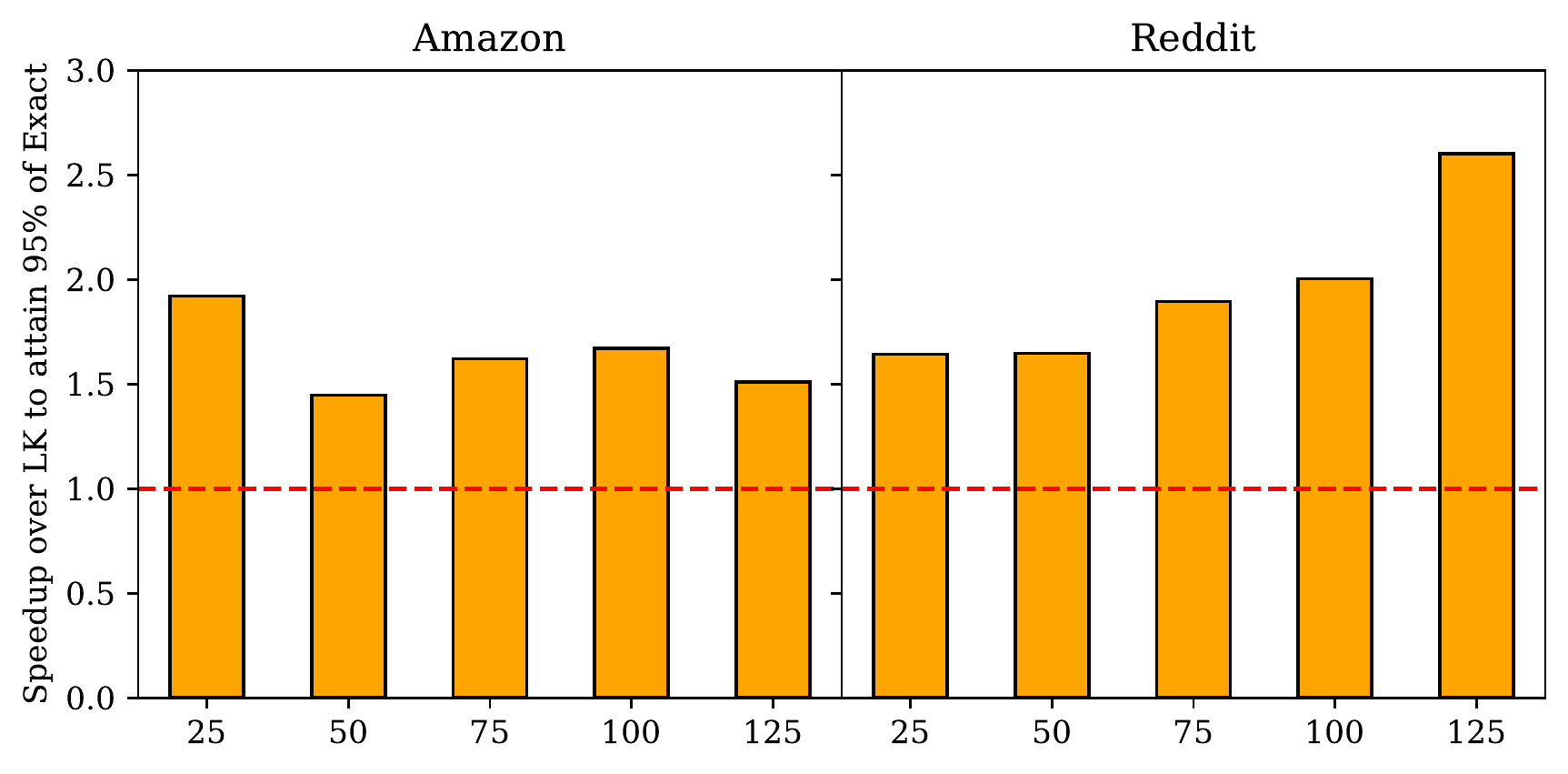}
         \caption{Large Tensors}
         \label{fig:large_tensor_speedup}
     \end{subfigure}
     \hfill
     \begin{subfigure}[b]{0.48\textwidth}
         \centering
         \includegraphics[scale=0.37]{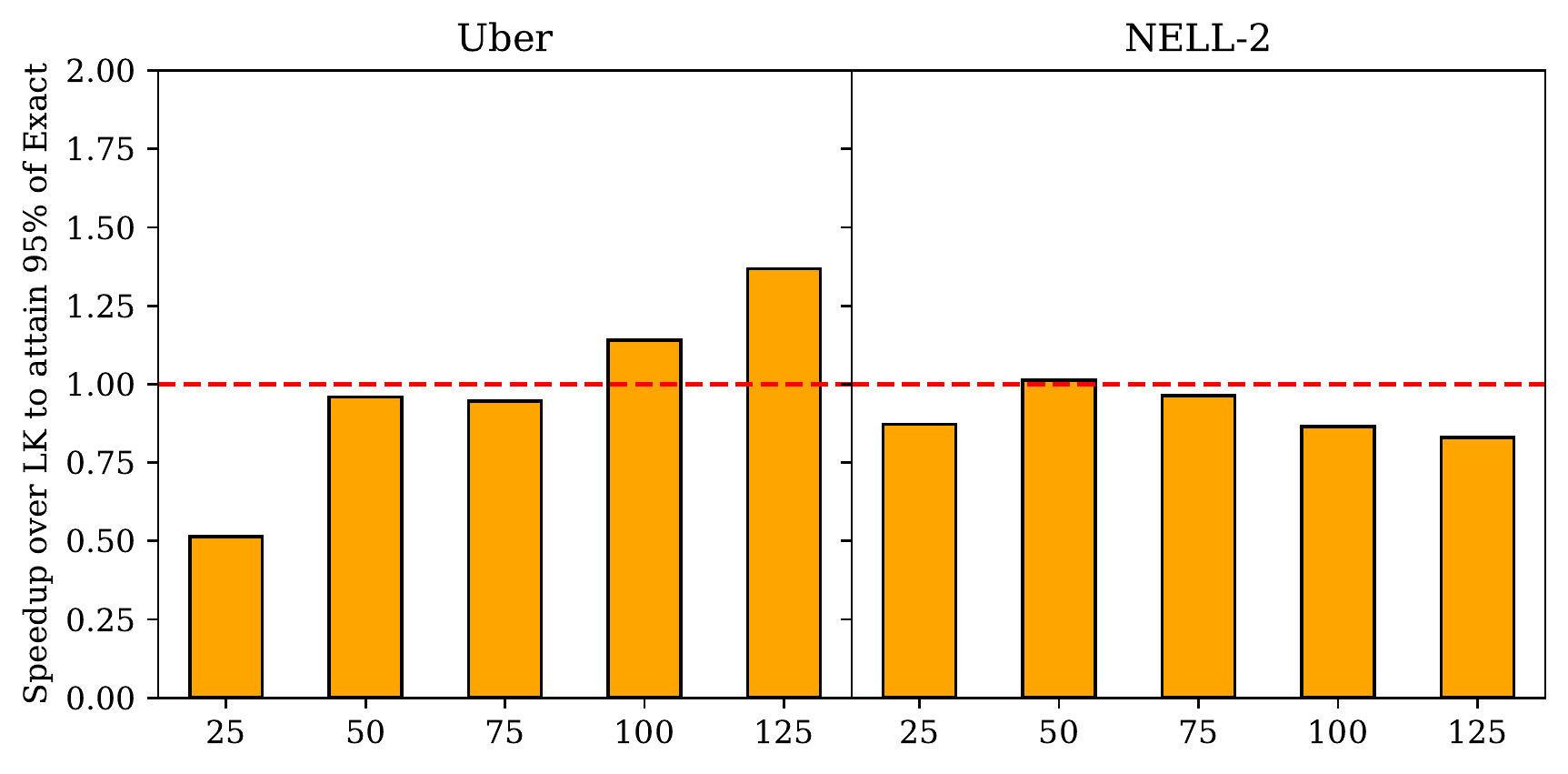}
         \caption{Small Tensors}
        \label{fig:small_tensor_speedup}
     \end{subfigure}
        \caption{Speedup of STS-CP over CP-ARLS-LEV hybrid
        (LK) to reach 95\% of the fit achieved by
        non-randomized ALS. Large tensors have more than
        1 billion nonzero entries.}
    \label{fig:tensor_speedup_no_enron}
\end{figure}

On the Enron tensor, hybrid CP-ARLS-LEV could not achieve 
the 95\% accuracy threshold for any rank above 25 
for the sample counts tested
in Table \ref{tab:cp_arls_lev_sample_counts}. 
\textbf{STS-CP achieved the threshold accuracy for 
all ranks tested}. Instead,
Figure \ref{fig:enron_speedup} reports the speedup
to achieve 85\% of the fit of non-randomized ALS
on the Enron. Beyond rank 25, our method consistently 
exhibits more than 2x speedup to reach the threshold.

\begin{figure}[ht]
    \centering
    \includegraphics[scale=0.4]{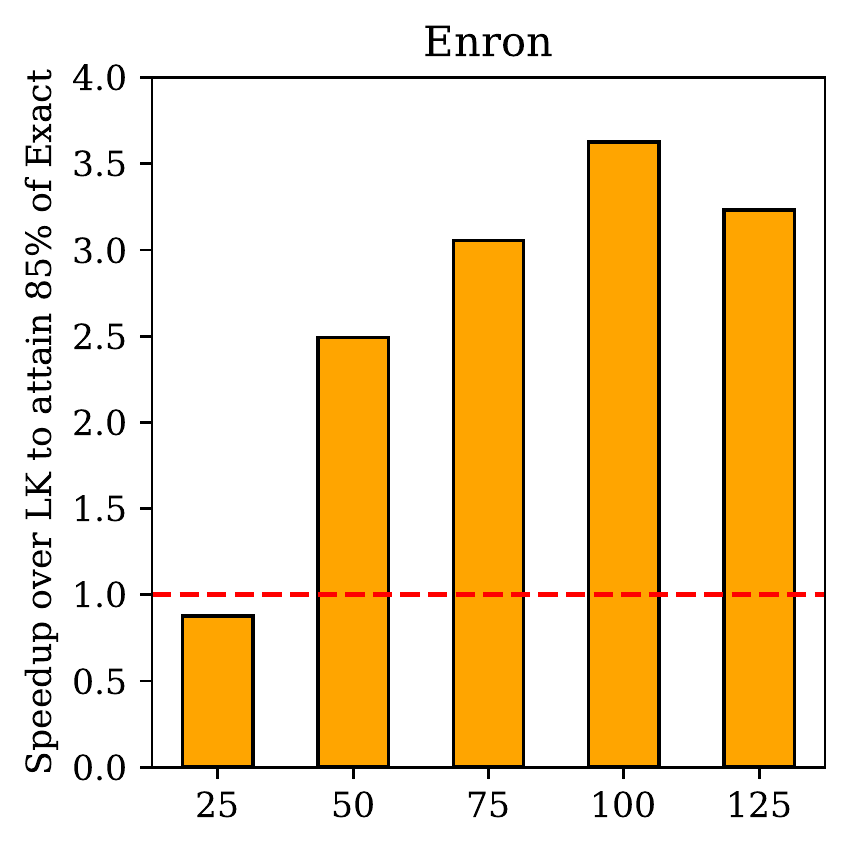}
    \caption{Speedup of STS-CP over CP-ARLS-LEV hybrid (LK)
    to reach 85\% of the fit achieved by
    non-randomized ALS on the Enron Tensor.}
    \label{fig:enron_speedup}
\end{figure}

\paragraph{Guide to Sampler Selection.}
Based on the performance comparisons in this section, we offer
the following guide to CP decomposition algorithm selection. 
Our experiments demonstrate that \textbf{STS-CP offers the 
most benefit on sparse tensors with billions of nonzeros 
(Amazon and Reddit) at high target decomposition rank}. Here, the 
runtime of our more expensive 
sampling procedure is 
offset by reductions in the least squares solve time. For 
smaller tensors, our sampler may still offer significant performance 
benefits (Enron). In other cases (Uber, NELL-2), CP-ARLS-LEV exhibits 
better performance, but by small margins for rank beyond 50.

STS-CP reduces the cost of each least squares solve
through a sample selection process that relies on dense 
linear algebra primitives (see Algorithms 
\ref{alg:row_sampler_construction} and \ref{alg:row_sampler_sampling}).
Because these operations can be expressed as standard
BLAS calls and can be carried out in parallel (see 
Appendix \appendixref{appendix:efficient_parallel_impl}, we
hypothesize that STS-CP is favorable when GPUs or other 
dense linear algebra accelerators are available.

Because our target tensor is sparse, the least squares solve
during each ALS iteration requires a sparse matricized-tensor
times Khatri-Rao product (spMTTKRP) operation. After sampling,
this primitive can reduced to sparse-matrix dense-matrix
multiplication (SpMM). Development of accelerators for
these primitives is an active area of 
research \cite{spMTTKRP_accelerator, spmm_accelerator}. When 
such accelerators are available, the lower cost of the spMTTKRP operation
reduces the relative benefit provided by the STS-CP sample selection
method. We hypothesize that CP-ARLS-LEV, with its faster 
sample selection process but lower sample efficiency, may retain
its benefit in this case. We leave verification of these two
hypotheses as future work.

\end{document}